\documentclass[a4paper,10pt]{article}
\usepackage[utf8]{inputenc}

\usepackage{cmap}
\usepackage[T1]{fontenc}

\usepackage{tocloft}

\usepackage[english]{babel}
\usepackage{amsfonts}
\usepackage{amssymb, amsthm}
\usepackage{xcolor,graphicx}
\usepackage{marginnote}
\usepackage{amsmath}
\usepackage{a4wide}
\usepackage[affil-it]{authblk}

\usepackage{epstopdf}
\usepackage{titlesec}
\titlelabel{\thetitle.\quad}

\usepackage[labelsep=period]{caption}

\usepackage{enumitem}

\numberwithin{equation}{section}

\let\OLDthebibliography\thebibliography
\renewcommand\thebibliography[1]{
	\OLDthebibliography{#1}
	\setlength{\parskip}{1pt}
	\setlength{\itemsep}{1pt plus 0.3ex}
}

\usepackage[colorlinks=true,linktocpage,pdfpagelabels,
bookmarksnumbered,bookmarksopen]{hyperref}
\definecolor{ForestGreen}{rgb}{0.1,0.6,0.05}
\definecolor{EgyptBlue}{rgb}{0.063,0.1,0.6}
\hypersetup{
	colorlinks=true,
	linkcolor=EgyptBlue,         
	citecolor=ForestGreen,
	urlcolor=olive
}

\usepackage[hyperpageref]{backref}

\allowdisplaybreaks
\usepackage{mathtools}
\mathtoolsset{showonlyrefs}

\usepackage{accents}

\newtheorem{theorem}{Theorem}[section]
\newtheorem{lemma}[theorem]{Lemma}
\newtheorem{proposition}[theorem]{Proposition}

\theoremstyle{definition}
\newtheorem{remark}[theorem]{Remark}

\newcommand{\W}{W_0^{1,p}(\Omega)}
\newcommand{\intO}{\int_\Omega}
\newcommand{\D}{\mathcal{D}_{\varphi_1}}

\title{Improved Friedrichs inequality for a subhomogeneous embedding
	\\ \medskip}

\author{Vladimir Bobkov ~\&~ Sergey Kolonitskii \\}
\date{}

\AtEndDocument{%
	\par
	\medskip
	\bigskip
	\begin{tabular}{@{}l@{}}%
		(V.~Bobkov)\\[0.2em]
		\textsc{Institute of Mathematics, Ufa Federal Research Centre, RAS}\\
		\textsc{Chernyshevsky str. 112, 450008 Ufa, Russia}\\[0.3em]
		\textit{E-mail address}: \texttt{bobkov@matem.anrb.ru}\\[1.5em]
		
		(S.~Kolonitskii)\\[0.2em]
		\textsc{Saint Petersburg Electrotechnical University ``LETI''}\\
		\textsc{Professora Popova str. 5, 197376 St. Petersburg, Russia}\\[0.3em]
		\textit{E-mail address}: \texttt{sergey.kolonitskii@gmail.com}
\end{tabular}}

\begin{document}
	
	\maketitle
	
	\vspace*{-5ex}
	\begin{abstract}
		For a smooth bounded domain $\Omega$ and $p \geqslant q \geqslant 2$, we establish quantified versions of the classical Friedrichs inequality $\|\nabla u\|_p^p - \lambda_1 \|u\|_q^p \geqslant 0$, $u \in W_0^{1,p}(\Omega)$, where $\lambda_1$ is a generalized least frequency. 
		We apply one of the obtained quantifications to show that the resonant equation $-\Delta_p u = \lambda_1 \|u\|_q^{p-q} |u|^{q-2} u + f$ coupled with zero Dirichlet boundary conditions possesses a weak solution provided $f$ is orthogonal to the minimizer of $\lambda_1$.

		\par
		\smallskip
		\noindent {\bf  Keywords}: 
		$p$-Laplacian; sublinear; improved Friedrichs inequality; improved Poincar\'e inequality;  Steklov inequality.
		
		\smallskip		
		\noindent {\bf MSC2010}: 
		35J92,	
		35P30,	
		35A23,	
		47J10.	
	\end{abstract}
	
	\begin{quote}	
		\tableofcontents	
		\addtocontents{toc}{\vspace*{-2ex}}
	\end{quote}

	\section{Introduction}\label{sec:intro}
	
	Let $p \geqslant q > 1$ and let $\Omega \subset \mathbb{R}^N$ be a domain of finite (Lebesgue) measure, $N \geqslant 1$.
	The embedding $W_0^{1,p}(\Omega) \hookrightarrow L^q(\Omega)$ is well known to be compact and 
	the best constant of this embedding can be characterized by means of the generalized least frequency 
	\begin{equation}\label{eq:lambda1}
		\lambda_1 
		=
		\inf
		\left\{
		\frac{\intO |\nabla u|^p \,dx}{\left(\intO |u|^q \,dx\right)^\frac{p}{q}}:~
		u \in W_0^{1,p}(\Omega) \setminus \{0\}
		\right\}.
	\end{equation}
	Let us recall several classical facts about $\lambda_1$ and its minimizers. 
	We have $\lambda_1>0$ and $\lambda_1$ is attained by a function $\varphi_1 \in \W$.
	In view of the subhomogeneity assumption $p \geqslant q$, the minimizer $\varphi_1$ is known to be \textit{unique} modulo scaling, see, e.g., \cite{otani2,kawohl,nazarov}.
	The regularity results \cite{diben,otani,talk} imply that $\varphi_1$ is bounded and belongs to $C^{1,\beta}_\text{loc}(\Omega)$ with some $\beta \in (0,1)$. 
	If, in addition, $\Omega$ is of class $C^{1,\alpha}$, $\alpha \in (0,1)$, then $\varphi_1 \in C^{1,\beta}(\overline{\Omega})$, see \cite{Lieberman}.
	By the strong maximum principle (see, e.g., \cite{PS2}), $\varphi_1$ cannot attain zero values in $\Omega$. Accordingly, we assume, without loss of generality, that $\varphi_1>0$ in $\Omega$. 	
	We refer to \cite{BrFr,FL} for a comprehensive overview of the properties of $\lambda_1$, $\varphi_1$, and other critical values and points of the Rayleigh-type quotient in \eqref{eq:lambda1}.
	
	\smallskip
	For functions from $\W$, the definition \eqref{eq:lambda1} leads to the standard \textit{Friedrichs inequality}\footnote{The name for the inequality might not be optimal from the historical perspective, but it is frequently used in the contemporary literature.
		We refer to \label{fot} \cite{KN} and \cite[Chapter II]{mitrin} for a comprehensive historical overview.}
	\begin{equation}\label{eq:Friedrichs}
		\intO |\nabla u|^p \,dx 
		- 
		\lambda_1 
		\left(\intO |u|^q \,dx\right)^\frac{p}{q}
		\geqslant 0,
	\end{equation}
	where the equality takes place if and only if $u$ is either a minimizer of $\lambda_1$ or trivial, i.e., $u \in \mathbb{R} \varphi_1$.
	We are interested in obtaining a quantification of the inequality \eqref{eq:Friedrichs}.
	In the homogeneous linear case $p=q=2$, in which \eqref{eq:Friedrichs} is often referred to as the \textit{Poincar\'e inequality}\textsuperscript{\ref{fot}}, 
	one easily sees that \eqref{eq:Friedrichs} can be refined via the spectral decomposition as
	\begin{equation}\label{eq:Friedrichs-linear}
		\intO |\nabla u|^2 \,dx - \lambda_1 \intO u^2 \,dx 
		\geqslant 
		\frac{\lambda_2-\lambda_1}{\lambda_2} \intO |\nabla u^\perp|^2 \,dx,
	\end{equation}
	where $\lambda_2$ is the second eigenvalue of the Dirichlet Laplacian in $\Omega$ and $u^\perp$ is the orthogonal projection in $L^2(\Omega)$ on the orthogonal complement of the eigenspace 
	$\mathbb{R}\varphi_1$. 
	In the homogeneous \textit{nonlinear} case $p=q > 2$, an improvement of the Poincar\'e inequality was obtained by \textsc{Fleckinger-Pell\'e} \& \textsc{Tak\'a\v{c}} in \cite{takac}. 
	Under certain regularity assumptions on $\Omega$, which will be discussed below, the authors proved that
	\begin{equation}\label{eq:improvedFriedrichs-Takac}
		\intO |\nabla u|^p \,dx 
		- 
		\lambda_1 
		\intO |u|^p \,dx
		\geqslant
		C\left(
		|\tilde{u}^\parallel|^{p-2} \intO |\nabla \varphi_1|^{p-2} |\nabla \tilde{u}^\perp|^2 \,dx + \intO |\nabla \tilde{u}^\perp|^p \,dx
		\right),
	\end{equation}
	where $\tilde{u}^\parallel \in \mathbb{R}$ and $\tilde{u}^\perp \in \W$ are defined from the $L^2(\Omega)$-decomposition $u = \tilde{u}^\parallel \varphi_1 + \tilde{u}^\perp$ by
	\begin{equation}\label{eq:decomp2-Takac}
		\tilde{u}^\parallel = \frac{\intO \varphi_1 u \, dx}{\intO \varphi_1^2 \,dx}
		\quad \text{and} \quad 
		\intO \varphi_1 \tilde{u}^\perp \,dx = 0.
	\end{equation}
	This result helps to provide fine estimates on function sequences and corresponding energy levels of various nonlinear problems. 
	In particular, \eqref{eq:improvedFriedrichs-Takac} 
	found its importance in the development of the Fredholm alternative for the $p$-Laplacian at $\lambda_1$, see the overviews \cite{drabek,takac-lec2,takac-lec1}, as well as in the investigation of other related equations, see, e.g., \cite{BT1,BobkovTanaka2021,CuDe,HalNid}. 
	We refer to \cite{AFT} and \cite{DKS} for versions of the improved Poincar\'e inequality \eqref{eq:improvedFriedrichs-Takac} in the entire space case and an exterior domain case, respectively, and to \cite{DrabTac} for the consideration of the case $p \in (1,2)$.
	In the case $p=2$ and $q > 2$, a quantification of a Friedrichs-type inequality for the function space  $W^{1,2}(\mathbb{S}^N)$ has been found in \cite{frank}. 
	
	\smallskip
	The aim of the present work is to obtain qualitative improvements of the nonlinear Friedrichs inequality
	\eqref{eq:Friedrichs} for $p \geqslant q \geqslant 2$
	in a form similar to \eqref{eq:Friedrichs-linear} and \eqref{eq:improvedFriedrichs-Takac}.
	For that purpose, we employ a different decomposition  than \eqref{eq:decomp2-Takac} which seems to be more suitable to work with nonhomogeneous problems. 
	Note that since $\varphi_1$ is a critical point of the Rayleigh quotient in \eqref{eq:lambda1}, it is a weak solution of the corresponding Lane-Emden problem:
	\begin{equation}\label{eq:varphi-solution}
		\intO |\nabla \varphi_1|^{p-2} \langle \nabla \varphi_1, \nabla v \rangle \,dx
		-
		\lambda_1 \left(\intO \varphi_1^q \,dx\right)^\frac{p-q}{q} \intO \varphi_1^{q-1} v \,dx = 0
		\quad \text{for all}~ v \in \W.
	\end{equation}
	Hereinafter, by $\langle \cdot,\cdot \rangle$ we denote the standard scalar product in $\mathbb{R}^N$.
	In view of the form of \eqref{eq:varphi-solution}, we decompose
	an arbitrary function $u \in \W$ as 
	\begin{equation}\label{eq:decomp}
		u = u^\parallel \varphi_1 + u^\perp,
	\end{equation}
	where $u^\parallel \in \mathbb{R}$ and $u^\perp \in \W$ are defined by
	\begin{equation}\label{eq:decomp2}
		u^\parallel = \frac{\intO \varphi_1^{q-1} u \, dx}{\intO \varphi_1^q \,dx}
		\quad \text{and} \quad 
		\intO \varphi_1^{q-1} u^\perp \,dx = 0.
	\end{equation}
	We also introduce the following mild regularity assumption on $\Omega$:
	\begin{enumerate}[label={($\mathcal{A}$)}]\addtolength{\itemindent}{0em}
		\item\label{A} In the case $N \geqslant 2$, $\Omega$ is a bounded domain of class $C^{1,\alpha}$ for some $\alpha \in (0,1)$. In the case $N=1$, $\Omega$ is a bounded open interval.
	\end{enumerate}
	
	\noindent
	Let us state our first main result on the improved Friedrichs inequality.
	\begin{theorem}\label{thm:main}
		Let $p \geqslant q \geqslant 2$ with $p>2$, and \ref{A} be satisfied.
		Then there exists $C=C(p,q,\Omega)>0$ such that
		\begin{equation}\label{eq:improvedFriedrichs}
			\intO |\nabla u|^p \,dx 
			- 
			\lambda_1 
			\left(\intO |u|^q \,dx\right)^\frac{p}{q}
			\geqslant
			C\left(
			|u^\parallel|^{p-2} \intO |\nabla \varphi_1|^{p-2} |\nabla u^\perp|^2 \,dx + \intO |\nabla u^\perp|^p \,dx
			\right)
		\end{equation}
		for any $u \in W_0^{1,p}(\Omega)$.
	\end{theorem}
	
	In the norm notation, the inequality \eqref{eq:improvedFriedrichs} can be written as
	\begin{equation}\label{eq:improvedFriedrichs-norm}
		\|\nabla u\|_p^p - \lambda_1 \|u\|_q^p 
		\geqslant
		C \left(
		|u^\parallel|^{p-2} \|u^\perp\|_{\varphi_1}^2 + \|\nabla u^\perp\|_p^p
		\right),
	\end{equation}
	where $\|\cdot\|_r$ denotes the usual norm in $L^r(\Omega)$, $r \in [1,\infty)$, i.e., $\|u\|_r^r = \int_\Omega |u|^r \, dx$, and the norm $\|\cdot\|_{\varphi_1}$ is introduced in \eqref{eq:normD} below.
	Lemma~\ref{lem:embed}~\ref{lem:embed:0} can be used to estimate $\|u^\perp\|_{\varphi_1}$ from below by $\|u^\perp\|_{\kappa}$ for some $\kappa > 2$, which might be convenient in applications.
	
	We remark that, in view of the decomposition \eqref{eq:decomp2}, the right-hand side of \eqref{eq:improvedFriedrichs}
	is $0$-homogeneous with respect to $\varphi_1$
	and $p$-homogeneous with respect to $u$.
	In particular, the constant $C$ does not depend on the normalization of $\varphi_1$.
	
	In the case $p=q=2$, the improved Friedrichs inequality \eqref{eq:improvedFriedrichs} coincides with \eqref{eq:Friedrichs-linear} up to a nonquantified constant $C$.
	In the case $p=q > 2$, \eqref{eq:improvedFriedrichs} has the same form as the improved Poincar\'e inequality \eqref{eq:improvedFriedrichs-Takac} from \cite{takac}, but we emphasize that \eqref{eq:improvedFriedrichs-Takac} is formulated in terms of the decomposition \eqref{eq:decomp2-Takac} which slightly differs from \eqref{eq:decomp2}, see Section~\ref{sec:general} for a further generalization.
	We also note that \cite{takac} asks for stronger regularity assumptions on $\Omega$ than \ref{A}, see \cite[(\textbf{H1}) and (\textbf{H2}), pp.~956-957]{takac}.
	While the assumption \cite[(\textbf{H1})]{takac} is essentially equivalent to \ref{A} (see a discussion in the proof of Lemma~\ref{lem:embed}~\ref{lem:embed:0} below), the assumption \cite[(\textbf{H2})]{takac} is less constructive and  was shown to be valid only when $N=1$ or when $\partial \Omega$ is connected (in the case $N \geqslant 2$). 
	At the same time, it was conjectured in \cite[Section~2.1]{takac-only} that \cite[(\textbf{H2})]{takac} is always satisfied provided \cite[(\textbf{H1})]{takac} holds.
	Recent embedding results from \cite{brasco-lind} (which are based on \cite{sciunz}) yield an affirmative answer to this conjecture, see Remark~\ref{rem:H2} below. 
	For convenience, we update the statement of \cite[Theorem~3.2]{takac} accordingly.
	\begin{theorem}[\cite{takac}]\label{thm:main-Takac}
		Let $p> 2$ and \ref{A} be satisfied.
		Then there exists $C=C(p,\Omega)>0$ such that \eqref{eq:improvedFriedrichs-Takac} holds
		for any $u \in W_0^{1,p}(\Omega)$.
	\end{theorem}	
	This result allows to weaken  assumptions on $\partial \Omega$ imposed, e.g., in \cite[Theorem~2.6]{BT1}, \cite[Theorem~1.11]{BobkovTanaka2021}, and in other related results, see a discussion in \cite[Remark~5]{BT1} and \cite[Remark~1.13]{BobkovTanaka2021}.
	
	Our proof of Theorem \ref{thm:main} is based on the strategy developed in \cite{takac}.
	Namely, we divide the consideration into two cases, by considering the following two function cones.
	For a fixed $\gamma \in (0,\infty)$,  we define a cone ``around'' $\mathbb{R}\varphi_1$:
	\begin{equation}\label{eq:coneC}
		\mathcal{C}_\gamma
		=
		\{ 
		u \in \W:~ \|\nabla u^\perp\|_p \leqslant \gamma |u^\parallel| \},
	\end{equation}
	and the complementary cone 
	\begin{equation}\label{eq:coneCprime}
		\mathcal{C}_\gamma'
		=
		\{ 
		u \in \W:~ \|\nabla u^\perp\|_p \geqslant \gamma |u^\parallel| \}.
	\end{equation}
	The proof of the inequality \eqref{eq:improvedFriedrichs} in the complementary cone $\mathcal{C}_\gamma'$, given in Section~\ref{sec:easy}, is rather straightforward thanks to the fact that functions from  $\mathcal{C}_\gamma'$ are ``far'' from $\mathbb{R}\varphi_1$.
	The proof in the cone $\mathcal{C}_\gamma$ is more subtle since we are ``close'' to $\mathbb{R}\varphi_1$, 
	and we deal with this case in Section~\ref{sec:hard}.
	It requires working with the linearization of the $p$-Laplacian, and we provide corresponding results in Section \ref{sec:prelim}.

	\subsection{Generalization}\label{sec:general} 
	Let us discuss a generalization of 
	Theorems~\ref{thm:main} and \ref{thm:main-Takac}, as well as of the inequality \eqref{eq:Friedrichs-linear}, suggested by the fact that the right-hand sides of the improved inequalities \eqref{eq:improvedFriedrichs-Takac} and \eqref{eq:improvedFriedrichs} are structurally similar. 
	It is not hard to observe that the decompositions given by \eqref{eq:decomp2-Takac} and \eqref{eq:decomp2} provide particular examples of the direct sum decomposition of $\W$:
	\begin{equation}\label{eq:decompW}
		\W 
		= \mathbb{R}\varphi_1 
		\oplus 
		\mathrm{Ker}(l),
	\end{equation}
	where $l:\W \to \mathbb{R}$ is a bounded linear functional  
	satisfying $l[\varphi_1] \neq 0$, and $\mathrm{Ker}(l)$ is its kernel.
	More precisely, in the case of \eqref{eq:decomp2-Takac}, we have $l[u]= \intO \varphi_1 u \,dx$, and in the case of \eqref{eq:decomp2}, we have $l[u]= \intO \varphi_1^{q-1} u \,dx$.
	Clearly, there are many similar functionals, and a particular choice may depend on an application.
	
	Our second main theorem generalizes Theorems~\ref{thm:main} and \ref{thm:main-Takac} for the decomposition \eqref{eq:decompW}.	
	\begin{theorem}\label{thm:gen}
		Let $p \geqslant q \geqslant 2$. 
		Let \ref{A} be satisfied provided $p>2$.
		Let $l:\W \to \mathbb{R}$ be a bounded linear functional such that $l[\varphi_1] = 1$, and 
		let $P:\W \to \mathrm{Ker}(l)$ be the corresponding projection operator:
		\begin{equation*}
			Pu = u - l[u]\varphi_1.
		\end{equation*}
		Then there exists $C=C(p,q,\Omega,l)>0$ such that
		\begin{equation}\label{eq:improvedFriedrichs-gen}
			\|\nabla u\|_p^p - \lambda_1 \|u\|_q^p
			\geqslant
			C\left(
			\big| l[u]\big |^{p-2} \|Pu\|_{\varphi_1}^2
			+
			\|\nabla Pu\|_p^p
			\right)
		\end{equation}
		for any $u \in W_0^{1,p}(\Omega)$.	
	\end{theorem}
	
	The proof of Theorem~\ref{thm:gen} is placed in Section~\ref{sec:general:proof}. 
	It is worth noting that the dependence of the right-hand side of \eqref{eq:improvedFriedrichs-gen} on the exponent $q$ is, in general, hidden only in the
	constant $C$. 
	The normalization $l[\varphi_1] = 1$ is imposed solely for brevity.
	
	Let us also observe that in contrast to Theorems~\ref{thm:main} and \ref{thm:main-Takac} the formulation of Theorem~\ref{thm:gen} \textit{includes} the homogeneous linear case $p=q=2$. 
	In this case, \eqref{eq:improvedFriedrichs-gen} reads as
	\begin{equation}\label{eq:improvedPoinc1}
		\intO |\nabla u|^2 \,dx 
		- 
		\lambda_1 
		\intO u^2 \,dx
		\geqslant
		C \intO |\nabla Pu|^2 \,dx,
	\end{equation}	
	and this generalization of the improved Poincar\'e inequality \eqref{eq:Friedrichs-linear} does not seem straightforward to us since it covers the case of nonorthogonal decompositions of $W_0^{1,2}(\Omega)$.
	
	\subsection{Application to the nonlinear Fredholm alternative}
	
	In order to justify the utility of the improved Friedrichs inequality \eqref{eq:improvedFriedrichs}, we consider the following model example of the boundary value problem at resonance:
	\begin{equation}\label{eq:P}
		\left\{
		\begin{aligned}
			-\text{div}(|\nabla u|^{p-2} \nabla u) &= \lambda_1 \|u\|_q^{p-q} |u|^{q-2}u  + f 
			&&\text{in}\ \Omega, \\
			u&=0 &&\text{on}\ \partial \Omega,
		\end{aligned}
		\right.
	\end{equation}
	where $p > q \geqslant 2$ and $f \in (\mathcal{D}_{\varphi_1})^* \setminus \{0\}$.
	Here, $(\mathcal{D}_{\varphi_1})^*$ is the dual of $\mathcal{D}_{\varphi_1}$. 
	As a practical case, we see from Lemma~\ref{lem:embed}~\ref{lem:embed:0} that $L^\frac{\kappa}{\kappa-1}(\Omega) \subset (\mathcal{D}_{\varphi_1})^*$, where the inclusion is understood in the sense that for every $f \in L^\frac{\kappa}{\kappa-1}(\Omega)$ there exists $\bar{f} \in (\mathcal{D}_{\varphi_1})^*$ such that $\bar{f}[\xi] = \intO f \xi \,dx$ for all $\xi \in \mathcal{D}_{\varphi_1}$.
	
	In the case $p=q=2$, the classical Fredholm alternative asserts that \eqref{eq:P} possesses a (nonunique) solution \textit{if and only if} $f[\varphi_1]  = 0$.
	In the nonlinear case $p=q \neq 2$, the situation is drastically different. 
	In particular, the existence might occur for some $f$ with $f[\varphi_1] \neq 0$. 
	We refer to \cite{drabek,takac-lec2,takac-lec1} for an overview.
	Note also that the problem \eqref{eq:P} with $f \equiv 0$  is a nonlinear (but homogeneous) eigenvalue problem \eqref{eq:varphi-solution} corresponding to $\lambda_1$, see \cite{BrFr,FL}.
	
	The improved Friedrichs inequality \eqref{eq:improvedFriedrichs} allows to give a simple proof of the following existence result which can be regarded as a particular case of the generalized Fredholm alternative.
	\begin{theorem}\label{thm:existence}
		Let $p > q \geqslant 2$ and \ref{A} be satisfied.
		Let $f \in (\mathcal{D}_{\varphi_1})^* \setminus \{0\}$ be such that $f[\varphi_1] = 0$.
		Then \eqref{eq:P} possesses a weak solution.
	\end{theorem}
	The proof of Theorem \ref{thm:existence} is placed in Section~\ref{sec:application}.

	\subsection{Alternative improvements of the Friedrichs inequality}
	The inequalities 
	\eqref{eq:improvedFriedrichs-Takac}, \eqref{eq:improvedFriedrichs}, and \eqref{eq:improvedFriedrichs-gen} are not the only possible refinements of the Friedrichs (Poincar\'e, when $p=q$) inequality \eqref{eq:Friedrichs}.
	The following quantification was kindly indicated to us by \textsc{L.~Brasco} and it is a consequence of the enhanced hidden convexity  established in \cite{BPZ} (see \cite[Eq.~(2.10)]{BPZ} and also \cite[Eq.~(2), p.~178]{LL}). 
	We place the proof in Section~\ref{sec:Br}.
	\begin{theorem}\label{thm:Br}
		Let $p \geqslant q >1$ and $\Omega$ be a domain of finite measure.
		Then there exists $C=C(p,q,\Omega)>0$ such that
		\begin{equation}\label{eq:Pimpr}
			\|\nabla u\|_p^p - \lambda_1 \|u\|_q^p
			\geqslant
			C\,\frac{\|u\|_q^p}{\|\varphi_1\|_q^p} \, 
			\max_{t \in [0,1]}
			\left[
			t(1-t) \intO \mathcal{R}_p\left(\frac{\|\varphi_1\|_q}{\|u\|_q}|u|,\varphi_1;t\right) dx
			\right]
		\end{equation}	
		for any $u \in \W$, where
		\begin{equation}\label{eq:R}
			\mathcal{R}_p(v,w;t)
			=
			\left\{
			\begin{aligned}
				&\frac{|w \nabla v - v\nabla w|^p}{(1-t)v^p+t w^p},
				&&\text{if}~ p \geqslant 2,\\
				&\frac{\left(
					|w \nabla v|^2 + |v \nabla w|^2
					\right)^\frac{p-2}{2}|w \nabla v - v\nabla w|^2}{(1-t)v^p+t w^p},
				&&\text{if}~ 1 < p < 2.
			\end{aligned}
			\right.
		\end{equation}
		Moreover, in the homogeneous case $p=q$, the following simpler inequality holds:
		\begin{equation}\label{eq:improvedPoinc2}
			\intO |\nabla u|^p \,dx 
			- 
			\lambda_1 
			\intO |u|^p \,dx
			\geqslant
			C \intO \mathcal{R}_p(|u|,\varphi_1;1) \,dx.
		\end{equation}  
		In particular, if $p=q \geqslant 2$, then 
		\begin{equation}\label{eq:improvedPoinc3}
			\intO |\nabla u|^p \,dx 
			- 
			\lambda_1 
			\intO |u|^p \,dx
			\geqslant
			C \intO \left|\nabla \left(\frac{u}{\varphi_1}\right)\right|^p \varphi_1^p \,dx.
		\end{equation}
	\end{theorem}
	
	We note that \eqref{eq:improvedPoinc2} is not a direct corollary of \eqref{eq:Pimpr}.
	It can be observed that the right-hand sides of the inequalities \eqref{eq:Pimpr} and \eqref{eq:improvedPoinc2} are $0$-homogeneous with respect to $\varphi_1$ and $p$-homogeneous with respect to $u$.
	Moreover, they measure the distance between $\varphi_1$ and $u$, although in a different way than in the inequalities \eqref{eq:improvedFriedrichs-Takac}, \eqref{eq:improvedFriedrichs}, and \eqref{eq:improvedFriedrichs-gen}.
	Compare also \eqref{eq:Friedrichs-linear}, \eqref{eq:improvedPoinc1}, and \eqref{eq:improvedPoinc3} (with $p=2$).
	Applications of \eqref{eq:Pimpr} and \eqref{eq:improvedPoinc2} are yet to be found.
	
	\smallskip
	By sacrificing the optimality of $\lambda_1$ in \eqref{eq:Friedrichs}, another improvements of the Poincar\'e and Friedrichs inequalities can be found as refinements of the Hardy inequality, for instance, 
	\begin{equation}\label{eq:Poin-impr}
		\intO |\nabla u|^p \,dx 
		- 
		\widehat{\lambda}
		\left(\intO |u|^q \,dx\right)^\frac{p}{q}
		\geqslant
		\frac{(N-p)^p}{p^p} \intO \frac{|u|^p}{|x|^p} \,dx,
	\end{equation} 
	see \cite[Theorem~4.1 and Extension~4.3]{BrezisVaz} for $N \geqslant 3$, $p=2$, and $1<q <2^*$, and \cite[Theorem~1]{GGM} for $1<p<N$ and $1<q \leqslant p$.
	We also refer to \cite{BEL} and references therein for a discussion about similar inequalities with the right-hand side containing a term of the form $\intO \frac{|u|^p}{\text{dist}(x,\Omega)^p} \,dx$. 
	As was noted, the value $\widehat{\lambda}$ on the left-hand side of \eqref{eq:Poin-impr} and related inequalities is compelled to satisfy $\widehat{\lambda} > \lambda_1$, which can be seen by substituting a minimizer of $\lambda_1$ in \eqref{eq:Poin-impr}.
	Moreover, unlike \eqref{eq:improvedFriedrichs-Takac},  \eqref{eq:improvedFriedrichs}, \eqref{eq:improvedFriedrichs-gen}, and \eqref{eq:Pimpr}, the right-hand side of \eqref{eq:Poin-impr} does not measure the distance between $u$ and the minimizers of $\lambda_1$. 	
	This prevents the usage of this generalization of \eqref{eq:Friedrichs} for the investigation of fine properties of 
	nonlinear problems at the resonance $\lambda_1$. 
	At the same time, \eqref{eq:Poin-impr} and its relatives are effectively used in the consideration of other problems, see, e.g., \cite{BEL,BrezisVaz,GGM} and references therein.

	\section{Preliminaries}\label{sec:prelim}
	
	Throughout this section, we always assume that $p \geqslant q \geqslant 2$ with $p>2$, and that $\Omega$ is a domain of finite measure, unless otherwise explicitly stated. 
	
	\subsection{Quadratic form}\label{sec:quad1}
	Let us introduce a functional $J \in C^2(\W,\mathbb{R})$ associated with the Friedrichs inequality \eqref{eq:Friedrichs}:
	$$
	J[u] 
	=
	\frac{1}{p} \intO |\nabla u|^p \,dx
	-
	\frac{\lambda_1}{p} \left(\intO |u|^q \,dx\right)^\frac{p}{q}.
	$$
	We know from \eqref{eq:Friedrichs} that $J[u] \geqslant 0$ for all $u \in \W$.
	Moreover, since $\varphi_1$ is a unique (modulo scaling) minimizer of $\lambda_1$, the equality $J[u]=0$ holds if and only if $u=t \varphi_1$, $t \in \mathbb{R}$.
	
	The Gateaux derivative of $J[u]$ in a direction $v \in \W$ is given by
	$$
	D J[u](v)
	=
	\intO |\nabla u|^{p-2} \langle \nabla u, \nabla v \rangle \,dx
	-
	\lambda_1 \left(\intO |u|^q \,dx\right)^\frac{p-q}{q} \intO |u|^{q-2} u v \,dx.
	$$
	In particular, by \eqref{eq:varphi-solution}, we get
	$D J[\varphi_1](v)=0$.
	
	Let us now take any $v \in \W$ and consider the function 
	$f(s) = J[\varphi_1 + sv]$, $s \in \mathbb{R}$.
	Noting that this function is of class $C^2(\mathbb{R})$, we can apply the Taylor formula with the reminder in the integral form,
	$$
	f(1) = f(0) + f'(0) + \int_0^1 f''(s) (1-s) \,ds,
	$$
	to get
	\begin{equation}\label{eq:JQ}
		J[\varphi_1 +v] = \int_0^1 D^2 J[\varphi_1 + sv](v,v) \, (1-s) \,ds,
	\end{equation}
	thanks to $f(0) = J[\varphi_1] = 0$ and $f'(0) =  DJ[\varphi_1](v) = 0$.
	The right-hand side of \eqref{eq:JQ} suggests to consider the following quadratic form for any $u,v \in W_0^{1,p}(\Omega)$:
	\begin{align*}
		Q_u(v,v)
		&=
		\int_0^1 D^2 J[\varphi_1 + su](v,v) \, (1-s) \,ds\\
		&=
		\int_0^1 \left(\intO\left<\mathbf{A}(\nabla \varphi_1+s\nabla u) \nabla v, \nabla v\right> dx \right) (1-s) \,ds
		\\
		&-
		\lambda_1
		(q-1)  \int_0^1 \left(\intO |\varphi_1 + su|^q \,dx\right)^\frac{p-q}{q}
		\left(\intO |\varphi_1+su|^{q-2} v^2
		\,dx \right) (1-s)\,ds\\
		&-
		\lambda_1
		(p-q) \int_0^1 \left(\intO |\varphi_1 + su|^q \,dx\right)^\frac{p-2q}{q}
		\left(\intO |\varphi_1+su|^{q-2} (\varphi_1+su) v \,dx \right)^2
		(1-s) \,ds.
	\end{align*}
	Here, $\mathbf{A}$ is a symmetric $N \times N$-matrix defined as
	\begin{equation}\label{eq:martxA}
		\mathbf{A}(a) = |a|^{p-2} \left(I + (p-2) \frac{a \otimes a}{|a|^2}\right),
		\quad a \in \mathbb{R}^N \setminus \{0\},
	\end{equation}
	where $a \otimes a$ is a matrix defined as $a \otimes a = (a_i a_j)_{i,j =1}^N$, and we set $\mathbf{A}(\vec{0})$ to be a zero matrix.
	The matrix $\mathbf{A}$ corresponds to the linearization of the $p$-Laplacian and it is not hard to see that
	\begin{equation}\label{eq:abb}
		|a|^{p-2} |b|^2 \leqslant 
		\langle \mathbf{A}(a)b, b\rangle \leqslant (p-1) |a|^{p-2} |b|^2
	\end{equation}
	for any $a,b \in \mathbb{R}^N$, see, e.g., \cite[Section~3]{takac-only}.
	
	Since we might have  $p<2q$, the following remark is necessary. 
	If $\intO |\varphi_1 + s_0u|^q \,dx = 0$ for some $u \in \W$ and $s_0 \in (0,1)$, then $u=-s_0^{-1} \varphi_1$ a.e.\ in $\Omega$.
	In this case, we have
	\begin{align}
		\notag
		&\int_0^1 \left(\intO |\varphi_1 + su|^q \,dx\right)^\frac{p-2q}{q}
		\left(\intO |\varphi_1+su|^{q-2} (\varphi_1+su) v \,dx \right)^2
		(1-s) \,ds
		\\
		\label{eq:int2}
		&=
		\left(\intO \varphi_1^q \,dx\right)^\frac{p-2q}{q}
		\left(\intO \varphi_1^{q-1} v \,dx \right)^2
		\int_0^1 |1-ss_0^{-1}|^{p-2} (1-s) \, ds,
	\end{align}
	where the integral over $s$ in  \eqref{eq:int2} is finite.
	That is, the quadratic form $Q_u(v,v)$ is well-defined for any $u,v \in W_0^{1,p}(\Omega)$, regardless the relation between $p$ and $2q$.
	
	Recalling from \eqref{eq:JQ} that $J[\varphi_1 +v] = Q_v(v,v)$, we deduce from \eqref{eq:Friedrichs} that
	\begin{equation}\label{eq:Qv>0}
		Q_v(v,v) \geqslant 0 \quad\text{for any}\quad v \in \W,
		\quad \text{and} \quad  
		Q_{\varphi_1}(\varphi_1,\varphi_1)=0.
	\end{equation}
	Due to the homogeneity, we have $Q_{tv}(v,v) = |t|^{-2} Q_{tv}(tv,tv) \geqslant 0$ for all $t \in \mathbb{R} \setminus \{0\}$, which yields, by the continuity,
	\begin{equation}\label{eq:Q0>0}
		Q_{0}(v,v) \geqslant 0 \quad\text{for any}\quad v \in \W,
		\quad \text{and} \quad 
		Q_0(\varphi_1,\varphi_1) = 0,
	\end{equation}
	where
	\begin{align}
		\notag
		Q_{0}(v,v)
		=
		\frac{1}{2}\intO\left<\mathbf{A}(\nabla \varphi_1) \nabla v, \nabla v\right> dx
		&-
		\frac{\lambda_1(q-1)}{2} \left(\intO \varphi_1^q \,dx\right)^\frac{p-q}{q}
		\intO \varphi_1^{q-2} v^2
		\,dx\\
		\label{eq:Q0}
		&-
		\frac{\lambda_1(p-q)}{2} \left(\intO \varphi_1^q \,dx\right)^\frac{p-2q}{q}
		\left(\intO \varphi_1^{q-1} v \,dx\right)^2.
	\end{align}

	\subsection{Weighted space and embeddings}
	Let us now discuss a natural domain for the quadratic form $Q_{0}$.
	We see from \eqref{eq:abb} that
	$$
	\|v\|_{\varphi_1}^2
	\leqslant
	\intO\left<\mathbf{A}(\nabla \varphi_1) \nabla v, \nabla v\right> dx
	\leqslant
	(p-1) \|v\|_{\varphi_1}^2
	$$
	for any $v \in \W$, 
	where
	\begin{equation}\label{eq:normD}
		\|v\|_{\varphi_1} 
		:=
		\left(
		\intO |\nabla \varphi_1|^{p-2} |\nabla v|^2 \,dx
		\right)^{\frac{1}{2}}.
	\end{equation}
	The seminorm \eqref{eq:normD} is actually a norm in $\W$, which follows, e.g., from the inequalities \eqref{eq:47takac} below, or from the fact that the critical set $\{x \in \Omega: |\nabla \varphi_1(x)| = 0\}$ has zero measure (see \cite{lou} or \cite{brasco-lind,sciunz}).
	Following \cite[Section 2.1]{takac-only}, we denote by $\D$ the completion of $\W$ with respect to this norm.
	Clearly, $\D$ is a Hilbert space with the scalar product induced by \eqref{eq:normD}.
	
	Since $C_0^\infty(\Omega)$ is dense in $\W$, it is also dense in $\mathcal{D}_{\varphi_1}$.
	Taking into account that the embedding $\D \hookrightarrow L^2(\Omega)$ is continuous (which follows from Lemma~\ref{lem:embed}~\ref{lem:embed:0} below), we conclude that $\D$ coincides with the spaces $X_0^{1,2}(\Omega; |\nabla \varphi_1|^{p-2})$ (in the notation of \cite{brasco-lind}) and $H_{0,\rho}^{1,2}(\Omega)$ with $\rho = |\nabla \varphi_1|^{p-2}$ (in the notation of \cite{sciunz}), which are defined as the completion of $C_0^{\infty}(\Omega)$ with respect to the norm $\|\cdot\|_2 + \|\cdot\|_{\varphi_1}$.
	
	In the following lemma, we collect several embedding results for the space $\D$ which are essentially based on \cite{brasco-lind,sciunz,FNSS,takac-only}. 
	We denote $r^* = \frac{Nr}{N-r}$ if $r<N$ and $r^*=\infty$ if $r \geqslant N$.
	\begin{lemma}\label{lem:embed}
		Let $p > 2$, $1<q<p^*$, and \ref{A} be satisfied.
		Then the following assertions hold:
		\begin{enumerate}[label={\rm(\roman*)}]
			\item\label{lem:embed:0} $\D \hookrightarrow L^\kappa(\Omega)$ compactly for some $\kappa \in (2,2^*)$;
			\item\label{lem:embed:2} $\D \hookrightarrow W_0^{1,\theta}(\Omega)$ continuously for some $\theta > 1$;
			\item\label{lem:embed:3} $\W \hookrightarrow \D$ continuously.
		\end{enumerate}	
	\end{lemma}
	
	In fact, in the proofs of our main theorems, we use the compactness of the embedding $\D \hookrightarrow L^\kappa(\Omega)$ only in the case $\kappa = 2$.
	However, we prove the more general statement, as it is interesting on its own.
	We place the proof of Lemma \ref{lem:embed} in Appendix \ref{sec:appendix}. 
	
	\begin{remark}\label{rem:H2}
		The continuous embedding $\D \hookrightarrow W_0^{1,\theta}(\Omega)$  guarantees that no function $u \in \D$ can be of the form $u = \varphi_1 \chi_S$ a.e.\ in $\Omega$, where $S \subset \Omega$ is a (Lebesgue) measurable set satisfying $0<|S|<|\Omega|$ and $\chi_S$ is the characteristic function of $S$, as it follows, e.g., from \cite[Theorem~2, p.~164]{EG}.
		In particular, this shows that the assumption \cite[(\textbf{H2})]{takac-only} (or, equivalently, \cite[(\textbf{H2})]{takac}) is always satisfied whenever \ref{A} holds, as conjectured in  \cite[Section~2.1]{takac-only}.		
	\end{remark}

	\subsection{Characterization of \texorpdfstring{$\lambda_1$}{lambda1} via the quadratic form}
	
	It can be shown exactly as in the proof of \cite[Proposition 3.5]{brasco-lind} that
	any sequence $\{v_n\} \subset C_0^\infty(\Omega)$ converging in $\mathcal{D}_{\varphi_1}$ to some $v \in \mathcal{D}_{\varphi_1}$ satisfies
	\begin{equation}\label{eq:conv-to-min}
		\lim\limits_{n \to \infty}\intO\left<\mathbf{A}(\nabla \varphi_1) \nabla v_n, \nabla v_n\right> dx
		=
		\intO\left<\mathbf{A}(\nabla \varphi_1) \nabla v, \nabla v\right> dx.
	\end{equation}
	Therefore, recalling that $C_0^\infty(\Omega)$ is dense in $\W$ and $\W$ is dense in $\mathcal{D}_{\varphi_1}$, 
	and using the continuity of the embedding 
	$\D \hookrightarrow L^2(\Omega)$ which follows from Lemma~\ref{lem:embed}~\ref{lem:embed:0} (under the assumption~\ref{A}), 
	we deduce from \eqref{eq:Q0>0} that 
	\begin{equation}\label{eq:Q0>0D}
		Q_{0}(v,v) \geqslant 0 \quad\text{for any}~ v \in \mathcal{D}_{\varphi_1}.
	\end{equation}
	Define the following critical value:
	\begin{equation}\label{eq:mu1}
		\mu_1
		=
		\inf_{v \in \mathcal{D}_{\varphi_1} \setminus \{0\}}
		\left\{
		\frac{\intO\left<\mathbf{A}(\nabla \varphi_1) \nabla v, \nabla v\right> dx}{(q-1)\left(\intO \varphi_1^q \,dx\right)^\frac{p-q}{q}
			\intO \varphi_1^{q-2} v^2
			\,dx
			+
			(p-q)\left(\intO \varphi_1^q \,dx\right)^\frac{p-2q}{q}
			\left(\intO \varphi_1^{q-1} v \,dx\right)^2}
		\right\}.
	\end{equation}
	The inequality \eqref{eq:Q0>0D} and the definition \eqref{eq:Q0} of $Q_{0}(v,v)$ immediately yield $\mu_1=\lambda_1$, which is therefore an alternative characterization of $\lambda_1$.
	We state this result explicitly.
	\begin{lemma}\label{lem:mu=lambda}
		Let $p \geqslant q \geqslant 2$ with $p>2$, and \ref{A} be satisfied. 
		Then
		$\mu_1 = \lambda_1$.
	\end{lemma}

	One of the most essential parts in our analysis is the following nondegeneracy result showing that the functions $u = t\varphi_1$ (with $t \in \mathbb{R}$) are the only elements of $\mathcal{D}_{\varphi_1}$ satisfying $Q_0(u,u)=0$. 
	In the case $p=q>2$, this result is given by \cite[Lemma 5.2]{takac} (or, equivalently,  \cite[Proposition 4.4]{takac-only}) in combination with Remark \ref{rem:H2}, and by \cite[Proposition 3.5]{brasco-lind}.
	\begin{lemma}\label{lem:minimizers}
		Let $p \geqslant q \geqslant 2$ with $p>2$, and \ref{A} be satisfied.
		Then the set of minimizers of $\mu_1$ is exhausted by $t \varphi_1$ with $t \neq 0$.
	\end{lemma}
	\begin{proof}
		Thanks to the references provided above, it is sufficient to consider only the case $p>q \geqslant 2$.
		We argue in a way similar to that from \cite[Proposition 3.5]{brasco-lind} and only the last part of our proof is different (in fact, simpler, thanks to $p>q$). 
		First, in view of \eqref{eq:conv-to-min}, 
		$\mu_1$ can be equivalently characterized through the minimization over $C_0^\infty(\Omega)$ instead of $\mathcal{D}_{\varphi_1}$, i.e.,
		$$
		\mu_1
		=
		\inf_{v \in C_0^\infty(\Omega) \setminus \{0\}}
		\left\{
		\frac{\intO\left<\mathbf{A}(\nabla \varphi_1) \nabla v, \nabla v\right> dx}{(q-1)\left(\intO \varphi_1^q \,dx\right)^\frac{p-q}{q}
			\intO \varphi_1^{q-2} v^2
			\,dx
			+
			(p-q)\left(\intO \varphi_1^q \,dx\right)^\frac{p-2q}{q}
			\left(\intO \varphi_1^{q-1} v \,dx\right)^2}
		\right\}.
		$$
		Suppose, by contradiction, that there exists a minimizer $v \in \mathcal{D}_{\varphi_1}$ of $\mu_1$ such that $v \not\in \mathbb{R}\varphi_1$.
		From the properties of quadratic forms, it is not hard to observe that any linear combination of $\varphi_1$ and $v$ also minimizes $\mu_1$.
		Alternatively, one could directly calculate $Q_0(\varphi_1 + \delta v,\varphi_1 + \delta v)$ for $\delta \in \mathbb{R}$ by noting that the chain of equalities
		\begin{align}
			\notag
			&\intO\left<\mathbf{A}(\nabla \varphi_1) \nabla \varphi_1, \nabla w\right> dx
			=
			\intO\left<\mathbf{A}(\nabla \varphi_1) \nabla w, \nabla \varphi_1\right> dx
			\\
			\label{eq:PA}
			&=
			(p-1)
			\intO |\nabla \varphi_1|^{p-2} \langle \nabla \varphi_1, \nabla w \rangle \,dx
			=
			\lambda_1 (p-1) \left(\intO \varphi_1^q \,dx\right)^\frac{p-q}{q} \intO \varphi_1^{q-1} w \, dx
		\end{align}
		holds for any $w \in \W$,
		to deduce that $Q_0(\varphi_1 + \delta v,\varphi_1 + \delta v)=0$.
		As a consequence, 
		if $\intO \varphi_1^{q-1} v \,dx \neq 0$, then there exists $\delta_0 \neq 0$ such that $\intO \varphi_1^{q-1} (v+\delta_0 \varphi_1) \,dx = 0$.
		Therefore, 	
		we can assume, without loss of generality, that the minimizer $v$ satisfies 
		$Q_0(v,v)=0$ and 
		$\intO \varphi_1^{q-1} v \,dx=0$.
		
		Let $\{v_n\} \subset C_0^\infty(\Omega)$ be a sequence converging to $v$ in $\mathcal{D}_{\varphi_1}$.
		Using $v_n^2/\varphi_1$ as test functions in \eqref{eq:PA}, we apply the Picone inequality from \cite[Lemma A.1]{brasco-lind} to obtain
		\begin{align}
			\notag
			(p-1) \lambda_1 \left(\intO \varphi_1^q \,dx\right)^\frac{p-q}{q} \intO \varphi_1^{q-2} v_n^2 \, dx
			&=
			\intO\left<\mathbf{A}(\nabla \varphi_1) \nabla \varphi_1, \nabla \left(\frac{v_n^2}{\varphi_1}\right)\right> dx
			\\
			\label{eq:lem231}
			&\leqslant
			\intO\left<\mathbf{A}(\nabla \varphi_1) \nabla v_n, \nabla v_n \right> dx.
		\end{align}
		Let us now pass to the limit in \eqref{eq:lem231} as $n \to \infty$.
		The passage to limit 
		on the left-had side (i.e., the weak term in \eqref{eq:lem231}) can be performed thanks to Lemma~\ref{lem:embed}~\ref{lem:embed:0}, while the right-hand side of \eqref{eq:lem231} converges thanks to  \eqref{eq:conv-to-min}.
		Thus, since $v$ satisfies $Q_0(v,v)=0$ and $\intO \varphi_1^{q-1} v \,dx=0$, we conclude from \eqref{eq:Q0} and \eqref{eq:lem231} that
		\begin{align*}
			\lambda_1 (p-1)\left(\intO \varphi_1^q \,dx\right)^\frac{p-q}{q} \intO \varphi_1^{q-2} v^2 \, dx
			&\leqslant
			\intO\left<\mathbf{A}(\nabla \varphi_1) \nabla v, \nabla v \right> dx
			\\
			&=
			\lambda_1 (q-1)\left(\intO \varphi_1^q \,dx\right)^\frac{p-q}{q}
			\intO \varphi_1^{q-2} v^2
			\,dx.
		\end{align*}
		Since $\varphi_1>0$ in $\Omega$, we have $\intO \varphi_1^{q-2} v^2 \,dx > 0$, and hence we arrive at a contradiction to the assumption $p > q$.
	\end{proof}

	\section{Proof of Theorem \ref{thm:main}}
	As we discussed in Section~\ref{sec:intro}, the proof of Theorem \ref{thm:main} splits into two cases: when $u \in \mathcal{C}_\gamma$ (for a sufficiently small $\gamma>0$) and when $u \in \mathcal{C}_\gamma'$ (for any $\gamma>0$), where the cones $\mathcal{C}_\gamma$ and $\mathcal{C}_\gamma'$ are defined in \eqref{eq:coneC} and \eqref{eq:coneCprime}, respectively.
	The analysis in the cone $\mathcal{C}_\gamma$ is more subtle 
	than that in $\mathcal{C}_\gamma'$
	since functions from $\mathcal{C}_\gamma$ are ``close'' to the subspace $\mathbb{R}\varphi_1$, and we start the proof with this case.
	
	For convenience, throughout the proof, we denote by $C>0$ a universal constant whose value may vary from inequality to inequality, but its exact value is irrelevant for our purposes.

	\subsection{Proof in \texorpdfstring{$\mathcal{C}_\gamma$}{C-gamma}}\label{sec:hard}
	
	For brevity, we introduce the following notation:
	\begin{align}
		\label{eq:P0tv}
		\mathcal{P}_1(t,v)
		&=
		\int_0^1 \left(\intO\left<\mathbf{A}(\nabla \varphi_1+st\nabla v) \nabla v, \nabla v\right> dx \right) (1-s) \,ds,\\
		\notag
		\mathcal{P}_0(t,v)
		&=
		(q-1)  \int_0^1 \left(\intO |\varphi_1 + stv|^q \,dx\right)^\frac{p-q}{q}
		\left(\intO |\varphi_1+stv|^{q-2} v^2
		\,dx \right) (1-s)\,ds\\
		\notag
		&+
		(p-q) \int_0^1 \left(\intO |\varphi_1 + stv|^q \,dx\right)^\frac{p-2q}{q}
		\left(\intO |\varphi_1+stv|^{q-2} (\varphi_1+stv) v \,dx \right)^2
		(1-s) \,ds.
	\end{align}
	Accordingly, in the case $t=0$, we have
	\begin{align}
		\label{eq:P00v}
		\mathcal{P}_1(0,v)
		&=
		\frac{1}{2}\intO\left<\mathbf{A}(\nabla \varphi_1) \nabla v, \nabla v\right> dx, \\
		\label{eq:P01v}
		\mathcal{P}_0(0,v)
		&=
		\frac{q-1}{2} \left(\intO \varphi_1^q \,dx\right)^\frac{p-q}{q}
		\intO \varphi_1^{q-2} v^2
		\,dx+
		\frac{p-q}{2} \left(\intO \varphi_1^q \,dx\right)^\frac{p-2q}{q}
		\left(\intO \varphi_1^{q-1} v \,dx \right)^2.~~
	\end{align}
	In view of the inequalities in \eqref{eq:Qv>0} and \eqref{eq:Q0>0}, the quadratic form $Q_{tv}(v,v)$ satisfies
	\begin{equation}\label{eq:QvPP}
		Q_{tv}(v,v) = \mathcal{P}_1(t,v) - \lambda_1 \mathcal{P}_0(t,v)
		\geq 0
		\quad \text{for all}~ t \in \mathbb{R},~ v \in \W.
	\end{equation}
	
	Observe that $\mathcal{P}_0(t,v)>0$ for any $t \in \mathbb{R}$ and $v \in \W \setminus \{0\}$, cf.\ \eqref{eq:int2}.
	We will also need the following inequality given by \cite[Lemma~A.2]{takac} (see also \cite[(5.3)]{takac}):
	there exists $C>0$ such that
	\begin{align}\label{eq:P1est}
		C 
		\left(
		\intO |\nabla \varphi_1|^{p-2} |\nabla v|^2 \,dx + |t|^{p-2}\intO |\nabla v|^p \,dx
		\right)
		&\leqslant 
		\mathcal{P}_1(t,v)
		\quad \text{for all}~ t \in \mathbb{R}, v \in \W.
	\end{align}
	
	For any $\gamma>0$, we consider the minimization problem
	\begin{equation}\label{eq:tilde-lambda}
		\widetilde{\Lambda}_\gamma
		=
		\inf
		\left\{
		\frac{\mathcal{P}_1(1,v)}{\mathcal{P}_0(1,v)}:~
		v \in \W \setminus \{0\},~
		\|\nabla v\|_p \leqslant \gamma,~ \intO \varphi_1^{q-1} v \,dx =0
		\right\}.
	\end{equation}
	It is readily seen from the inequality in \eqref{eq:QvPP} that $\lambda_1 \leqslant \widetilde{\Lambda}_\gamma$.
	Moreover, $\widetilde{\Lambda}_{\gamma_1} \leqslant \widetilde{\Lambda}_{\gamma_2}$ provided $\gamma_2 \leqslant \gamma_1$, thanks to the ``monotonicity'' of the constraint $\|\nabla v\|_p \leqslant \gamma$ with respect to $\gamma$.
	The main property of $\widetilde{\Lambda}_\gamma$ needed to prove Theorem~\ref{thm:main} in $\mathcal{C}_\gamma$
	consists in the fact that $\lambda_1 < \widetilde{\Lambda}_\gamma$ for any sufficiently small $\gamma>0$. 
	More precisely, we have the following separation result.
	\begin{proposition}\label{prop:lambda<Lambda}
		Let $p \geqslant q \geqslant 2$ with $p>2$, and \ref{A} be satisfied.
		Then there exists $\gamma_0 >0$ such that 
		$\lambda_1 < \widetilde{\Lambda}_{\gamma_0}$.
	\end{proposition}
	\begin{proof}
		Suppose, by contradiction, that 
		there exists a decreasing sequence $\{\gamma_n\} \subset (0,\infty)$ such that $\gamma_n \to 0$ and  $\lambda_1 = \widetilde{\Lambda}_{\gamma_n}$. 
		Let $\{v_n\} \subset \W \setminus\{0\}$ be a sequence which satisfies $\|\nabla v_n\|_p \to 0$,  $\intO \varphi_1^{q-1} v_n \,dx =0$, 
		and
		\begin{equation}\label{eq:convP1P2toL}
			\frac{\mathcal{P}_1(1,v_n)}{\mathcal{P}_0(1,v_n)} \to \lambda_1
			\quad \text{as}~ n \to \infty.
		\end{equation}
		The existence of such a sequence follows from the diagonal argument performed over minimizing sequences of $\widetilde{\Lambda}_{\gamma_n}$ along $n$.
		
		We pass to the normalized sequence $\{w_n\}$ consisting of functions $w_n = v_n/t_n$, where $t_n \in \mathbb{R} \setminus \{0\}$ is chosen to satisfy $\mathcal{P}_0(t_n,w_n)=1$.
		Such $t_n$ always exists since $0<\mathcal{P}_0(1,v_n)=t_n^2\mathcal{P}_0(t_n,w_n)$. 
		More precisely, $t_n^2 = \mathcal{P}_0(1,v_n)$.
		Noting that $\mathcal{P}_1(1,v_n)=t_n^2\mathcal{P}_1(t_n,w_n)$, we have $\mathcal{P}_1(t_n,w_n) \to \lambda_1$, and the ``orthogonality'' condition $\intO \varphi_1^{q-1} w_n \,dx =0$ is satisfied.
		
		Let us observe that $t_n \to 0$.
		Indeed, thanks to the convergence $\|\nabla v_n\|_p \to 0$,
		we apply the triangle inequality and the H\"older inequality to estimate $\mathcal{P}_0(1,v_n)$ from above in terms of $\|\nabla v_n\|_p$ and hence to deduce that $\mathcal{P}_0(1,v_n) \to 0$, which yields $t_n \to 0$.
		
		Since $\mathcal{P}_1(t_n,w_n) \to \lambda_1$, the inequality \eqref{eq:P1est} gives the boundedness of $\{w_n\}$ in $\mathcal{D}_{\varphi_1}$ and $\{t_n^\frac{p-2}{p} w_n\}$ in $\W$.
		Consequently, there exists $w_0 \in \D$ such that the 
		following convergences take place
		along a common subsequence of indices (see Lemma \ref{lem:embed} and the Rellich-Kondrachov theorem): 
		\begin{enumerate}[label={\rm(\roman*)}]
			\item\label{conv:1} $w_n \to w_0$ weakly in $\mathcal{D}_{\varphi_1}$ and $W_0^{1,\theta}(\Omega)$ for some $\theta >1$, and strongly in $L^2(\Omega)$.
			\item\label{conv:3} $t_n^\frac{p-2}{p} w_n \to 0$ weakly in $\W$ and strongly in $L^r(\Omega)$, $r \in (1,p^*)$. Here, the limit is zero since $t_n \to 0$ implies $t_n^\frac{p-2}{p} w_n \to 0$ strongly in $L^2(\Omega)$ by \ref{conv:1}.
			\item\label{conv:3x}
			$t_n^\frac{q-2}{q} w_n \to 0$  strongly in $L^q(\Omega)$. Indeed, applying the H\"older inequality and using \ref{conv:1} and \ref{conv:3}, we get
			\begin{align*}
				\|t_n^\frac{q-2}{q} w_n\|_{q}^q
				=
				\int_\Omega t_n^{q-2} |w_n|^q \,dx
				&\leq
				\left(
				\int_\Omega t_n^{p-2}|w_n|^p \,dx
				\right)^\frac{q-2}{p-2}
				\left(
				\int_\Omega w_n^2 \,dx
				\right)^\frac{p-q}{p-2}
				\\
				&=
				\|t_n^\frac{p-2}{p} w_n\|_p^\frac{p(q-2)}{p-2}
				\|w_n\|_2^\frac{2(p-q)}{p-2} \to 0.
			\end{align*}
			\item\label{conv:2} $t_n w_n \to 0$ strongly in $\W$ and $L^r(\Omega)$, $r \in (1,p^*)$, and a.e.\ in $\Omega$, since $v_n = t_nw_n$ and $\|\nabla v_n\|_p \to 0$.			
		\end{enumerate} 
		Moreover, we have $\intO \varphi_1^{q-1} w_0 \,dx =0$, thanks to \ref{conv:1}.

		Our aim now is to prove the following two facts:
		\begin{equation}\label{eq:convP0}
			\lim\limits_{n \to \infty}\mathcal{P}_0(t_n,w_n) = \mathcal{P}_0(0,w_0) = 1
		\end{equation}
		and
		\begin{equation}\label{eq:convP1}
			\liminf\limits_{n \to \infty}\mathcal{P}_1(t_n,w_n) \geqslant \mathcal{P}_1(0,w_0).
		\end{equation}
		We start with the convergence \eqref{eq:convP0} and recall that $\mathcal{P}_0(t_n,w_n)=1$. 
		Since $\intO \varphi_1^{q-1} w_0 \,dx =0$, we see from \eqref{eq:P01v} that
		\begin{align*}
			\mathcal{P}_0(0,w_0)
			=
			\frac{q-1}{2} \left(\intO \varphi_1^q \,dx\right)^\frac{p-q}{q}
			\intO \varphi_1^{q-2} w_0^2
			\,dx.
		\end{align*}
		At the same time, in view of the $L^q(\Omega)$-convergence given in \ref{conv:2}, the triangle inequality yields
		\begin{equation}\label{eq:phi1qconv}
			\left(\intO |\varphi_1 + st_n w_n|^q \,dx\right)^\frac{1}{q}
			\to 
			\left(\intO \varphi_1^q \,dx\right)^\frac{1}{q} > 0
		\end{equation}
		uniformly with respect to $s \in [0,1]$.
		Therefore, in order to establish \eqref{eq:convP0}, it is sufficient to prove that
		\begin{equation}\label{eq:convP0-1}
			\intO |\varphi_1+st_n w_n|^{q-2} w_n^2 \,dx
			\to 
			\intO \varphi_1^{q-2} w_0^2 \,dx
		\end{equation}
		and
		\begin{equation}\label{eq:convP0-2}
			\intO |\varphi_1+st_n w_n|^{q-2} (\varphi_1+s t_n w_n) w_n \,dx
			\to 0,
		\end{equation}
		both convergences being uniform with respect to $s \in [0,1]$.
		
		First, we justify \eqref{eq:convP0-1}.
		If $q=2$, then the convergence is given by \ref{conv:1}.
		So, assume that $q>2$.
		We have
		\begin{align}
			\notag
			&\left|\intO |\varphi_1+st_n w_n|^{q-2} w_n^2 \,dx
			-
			\intO \varphi_1^{q-2} w_0^2 \,dx\right|
			\\
			\label{eq:convP0-1:1}
			&\leqslant
			\intO \left| |\varphi_1+st_n w_n|^{q-2} - \varphi_1^{q-2}\right| w_n^2 \,dx
			+
			\intO \varphi_1^{q-2} \left|w_n^2 - w_0^2\right| \,dx.
		\end{align}
		The second integral on the right-hand side of \eqref{eq:convP0-1:1} converges to zero by \ref{conv:1} and the fact that $\varphi_1$ is bounded in $\Omega$. 
		Consider the first integral on the right-hand side of \eqref{eq:convP0-1:1}. 
		In view of the a.e.-convergence in \ref{conv:2}, Egorov's theorem asserts that for any $\varepsilon>0$ we can find a measurable subset $E_\varepsilon \subset \Omega$ such that $|E_\varepsilon| < \varepsilon$ and $t_n w_n \to 0$ uniformly in $\Omega \setminus E_\varepsilon$ and thus strongly in $L^\infty(\Omega \setminus E_\varepsilon)$.
		Consequently, $|\varphi_1+st_n w_n|^{q-2} \to \varphi_1^{q-2}$ strongly in $L^\infty(\Omega \setminus E_\varepsilon)$ and, clearly, uniformly with respect to $s \in [0,1]$.
		Therefore, decomposing
		\begin{align}
			\notag
			\intO \left| |\varphi_1+st_n w_n|^{q-2} - \varphi_1^{q-2}\right| w_n^2 \,dx
			&=
			\int_{\Omega \setminus E_\varepsilon} \left| |\varphi_1+st_n w_n|^{q-2} - \varphi_1^{q-2}\right| w_n^2 \,dx
			\\
			\label{eq:convP0-11}
			&+
			\int_{E_\varepsilon} \left| |\varphi_1+st_n w_n|^{q-2} - \varphi_1^{q-2}\right| w_n^2 \,dx
		\end{align}
		and recalling that $\{w_n\}$ is bounded in $L^2(\Omega)$ by \ref{conv:1}, we deduce that the first integral on the right-hand side of \eqref{eq:convP0-11} converges to zero uniformly with respect to $s \in [0,1]$.
		As for the second integral on the right-hand side of \eqref{eq:convP0-11}, we estimate it roughly as follows:
		\begin{equation}\label{eq:I2}
			\int_{E_\varepsilon} \left| |\varphi_1+st_n w_n|^{q-2} - \varphi_1^{q-2}\right| w_n^2 \,dx
			\leqslant 
			C  \int_{E_\varepsilon} \varphi_1^{q-2} w_n^2 \,dx
			+ 
			C \, s \int_{E_\varepsilon}
			t_n^{q-2} |w_n|^q \,dx.
		\end{equation}
		Since $\varphi_1$ is bounded in $\Omega$ and $\{w_n\}$ converges in $L^2(\Omega)$ by \ref{conv:1},
		we get
		$$
		\int_{E_\varepsilon} \varphi_1^{q-2} w_n^2 \,dx 
		=
		\int_{E_\varepsilon} \varphi_1^{q-2} w_0^2 \,dx + o(1),
		$$
		and $\int_{E_\varepsilon} \varphi_1^{q-2} w_0^2 \,dx$ converges to zero as $\varepsilon \to 0$ by the absolute continuity of the Lebesgue integral.
		The second integral on the right-hand side of \eqref{eq:I2} converges to zero as $n \to \infty$ thanks to \ref{conv:3x}.
		Summarizing, we obtain the desired convergence \eqref{eq:convP0-1} by successively passing to the limit as $n \to \infty$ and then as $\varepsilon \to 0$.
		
		Let us now discuss the convergence \eqref{eq:convP0-2}. 
		Decomposing
		\begin{equation}\label{eq:convP0-2:1}
			\intO |\varphi_1+st_n w_n|^{q-2} (\varphi_1+s t_n w_n) w_n \,dx
			=
			\intO |\varphi_1+st_n w_n|^{q-2} \varphi_1 w_n \,dx
			+
			s t_n \intO |\varphi_1+st_n w_n|^{q-2} w_n^2 \,dx,
		\end{equation}
		we see that the second integral on the right-hand side of \eqref{eq:convP0-2:1} converges to zero in view of \eqref{eq:convP0-1} since $t_n \to 0$.
		The convergence to zero of the first integral on the right-hand side of \eqref{eq:convP0-2:1} can be established in much the say way as above, by appealing to Egorov's theorem and the absolute continuity of the Lebesgue integral. We omit details.
		Combining the convergences \eqref{eq:phi1qconv}, \eqref{eq:convP0-1}, and \eqref{eq:convP0-2}, we finish the proof of the convergence \eqref{eq:convP0}.
		Also, we observe that \eqref{eq:convP0} implies $w_0 \not\equiv 0$ a.e.\ in $\Omega$. 
		
		\smallskip
		Let us turn to the justification of the weak lower semicontinuity type
		inequality \eqref{eq:convP1}:
		\noeqref{eq:convP111}
		\begin{equation}\label{eq:convP111}
			\tag{\ref{eq:convP1}}
			\liminf\limits_{n \to \infty}\mathcal{P}_1(t_n,w_n) \geqslant \mathcal{P}_1(0,w_0).
		\end{equation}
		Expanding, for convenience, $\mathcal{P}_1(t_n,w_n)$ and $\mathcal{P}_1(0,w_0)$ given by \eqref{eq:P0tv} and \eqref{eq:P00v}, respectively, we have
		\begin{align*}
			\mathcal{P}_1(t_n,w_n)
			&=
			\int_0^1
			\bigg[
			\intO|\nabla \varphi_1 + s t_n \nabla w_n|^{p-2} |\nabla w_n|^2 \,dx
			\\
			&+
			(p-2)
			\intO|\nabla \varphi_1 + s t_n \nabla w_n|^{p-4} \left<\nabla \varphi_1 + s t_n \nabla w_n, \nabla w_n \right>^2 \,dx
			\bigg] (1-s) \,ds
		\end{align*}
		and
		\begin{align*}
			\mathcal{P}_1(0,w_0)
			&=
			\frac{1}{2}
			\intO|\nabla \varphi_1|^{p-2} |\nabla w_0|^2 \,dx
			+
			\frac{p-2}{2}
			\intO|\nabla \varphi_1|^{p-4} \left<\nabla \varphi_1, \nabla w_0 \right>^2 \,dx.
		\end{align*}
		Thanks to Fatou's lemma and the superadditivity of the limit inferior, 
		the validity of \eqref{eq:convP1} will be implied by the separate validity of the following two inequalities for any fixed $s \in [0,1]$:
		\begin{equation}\label{eq:convergence1}
			\liminf_{n \to \infty}
			\intO|\nabla \varphi_1 + s t_n \nabla w_n|^{p-2} |\nabla w_n|^2 \,dx
			\geqslant
			\intO|\nabla \varphi_1|^{p-2} |\nabla w_0|^2 \,dx
		\end{equation}
		and
		\begin{equation}\label{eq:convergence2}
			\liminf_{n \to \infty}
			\intO|\nabla \varphi_1 + s t_n \nabla w_n|^{p-4} \left<\nabla \varphi_1 + s t_n \nabla w_n, \nabla w_n \right>^2 \,dx
			\geqslant
			\intO|\nabla \varphi_1|^{p-4} \left<\nabla \varphi_1, \nabla w_0 \right>^2 \,dx.
		\end{equation}
		
		Prior to the proof of \eqref{eq:convergence1} and \eqref{eq:convergence2}, let us observe that the weak convergence $w_n \to w_0$ in $W_0^{1,\theta}(\Omega)$ stated in \ref{conv:1} implies the weak convergence $\nabla w_n \to \nabla w_0$ in $L^\theta(\Omega;\mathbb{R}^N)$, since the operator $T:W_0^{1,\theta}(\Omega) \to L^\theta(\Omega;\mathbb{R}^N)$ defined as $T(u) = \nabla u$ is  linear and bounded. 
		As a consequence, $\nabla w_n - \nabla w_0 \to \vec{0}$ weakly in $L^\theta(\Omega\setminus E; \mathbb{R}^N)$ for any measurable set $E \subset \Omega$. 
		
		Also, it follows from \ref{conv:2} that $t_n \nabla w_n \to \vec{0}$ strongly in $L^p(\Omega;\mathbb{R}^N)$ and hence a.e.\ in $\Omega$.
		Thus, by Egorov's theorem, for any $\varepsilon>0$ there exists a measurable subset $E_\varepsilon \subset \Omega$ such that $|E_\varepsilon|< \varepsilon$ and $t_n \nabla w_n \to \vec{0}$ uniformly in $\Omega \setminus E_\varepsilon$ and thus strongly in $L^\infty(\Omega\setminus E_\varepsilon;\mathbb{R}^N)$.
		This yields
		\begin{equation}\label{eq:unifconv}
			\nabla \varphi_1 + s t_n \nabla w_n \to \nabla \varphi_1
			\quad \text{strongly in}~ L^\infty(\Omega\setminus E_\varepsilon; \mathbb{R}^N).
		\end{equation}
		Let us define
		\begin{equation*}
			E_{\varepsilon,k} = E_\varepsilon \cup \lbrace x \in \Omega:\, |\nabla w_0| \geqslant k \rbrace.
		\end{equation*}
		We see that $\nabla w_0$ is bounded in $\Omega \setminus E_{\varepsilon,k}$, and hence
		\begin{equation}\label{eq:conv-vec1}
			|\nabla \varphi_1 + s t_n \nabla w_n|^{p-2} \nabla w_0  \to 
			|\nabla \varphi_1|^{p-2} \nabla w_0  
			\quad \text{strongly in}~ L^\infty(\Omega\setminus E_{\varepsilon,k};\mathbb R^N),
		\end{equation}
		since $\Omega \setminus E_{\varepsilon,k} \subset \Omega \setminus E_{\varepsilon}$. 
		In particular, in view of the boundedness of $\Omega \setminus E_{\varepsilon,k}$, the H\"older inequality guarantees that the convergence \eqref{eq:conv-vec1} is also strong in $L^{\theta'}(\Omega\setminus E_{\varepsilon,k};\mathbb R^N)$.
		
		Let us now prove the inequality \eqref{eq:convergence1}.
		Since the function $a \mapsto |a|^2$ (with $a \in \mathbb{R}^N$) is convex, we have 
		\begin{equation}\label{eq:aconvex}
			|a_n|^2 \geqslant |a|^2 + 2 \langle a, a_n-a \rangle 
			\quad\text{for any}~ a_n, a \in \mathbb{R}^N.
		\end{equation}
		Substituting $a_n = \nabla w_n$ and $a = \nabla w_0$ into \eqref{eq:aconvex}, we obtain
		\begin{align}
			\notag
			&\int_{\Omega \setminus E_{\varepsilon,k}} |\nabla \varphi_1 + s t_n \nabla w_n|^{p-2} |\nabla w_n|^2 \, dx 
			\\
			\label{eq:sumtwoint1}
			&\geqslant 
			\int_{\Omega \setminus E_{\varepsilon,k}} |\nabla \varphi_1 + s t_n \nabla w_n|^{p-2} |\nabla w_0|^2 \, dx + 2 \int_{\Omega \setminus E_{\varepsilon,k}} |\nabla \varphi_1 + s t_n \nabla w_n|^{p-2} \langle \nabla w_0, \nabla w_n - \nabla w_0 \rangle \, dx.
		\end{align}
		The first integral on the right-hand side of \eqref{eq:sumtwoint1} converges to $\int_{\Omega \setminus E_{\varepsilon,k}} |\nabla \varphi_1|^{p-2} |\nabla w_0|^2 \, dx$ thanks to \eqref{eq:conv-vec1}.
		Recalling that the convergence \eqref{eq:conv-vec1} is strong also in  $L^{\theta'}(\Omega\setminus E_{\varepsilon,k};\mathbb R^N)$ and $\nabla w_n - \nabla w_0 \to \vec{0}$ weakly in $L^{\theta}(\Omega\setminus E_{\varepsilon,k};\mathbb R^N)$, we deduce that the second integral on the right-hand side of \eqref{eq:sumtwoint1} tends to zero. 
		Thus, we obtain the inequalities
		\begin{align*}
			&\liminf_{n\to \infty} \int_{\Omega} |\nabla \varphi_1 + s t_n \nabla w_n|^{p-2} |\nabla w_n|^2 \, dx
			\\
			&\geqslant
			\liminf_{n\to \infty} \int_{\Omega \setminus E_{\varepsilon,k}} |\nabla \varphi_1 + s t_n \nabla w_n|^{p-2} |\nabla w_n|^2 \, dx
			\geqslant 
			\int_{\Omega \setminus E_{\varepsilon,k}} |\nabla \varphi_1|^{p-2} |\nabla w_0|^2 \, dx.
		\end{align*}
		Passing successively to the limit as $k \to \infty$ and then as $\varepsilon \to 0$, 
		we conclude that \eqref{eq:convergence1} is satisfied.
		
		Let us now prove the inequality \eqref{eq:convergence2} using the same strategy as above.
		Define, for brevity, the vector-functions
		\begin{equation*}
			\Phi_n = |\nabla \varphi_1 + s t_n \nabla w_n|^{p-4} \left<\nabla \varphi_1 + s t_n \nabla w_n, \nabla w_0 \right> (\nabla \varphi_1 + s t_n \nabla w_n)
		\end{equation*}
		and
		\begin{equation*}
			\Phi = |\nabla \varphi_1|^{p-4} \left<\nabla \varphi_1, \nabla w_0 \right> \nabla \varphi_1.
		\end{equation*}
		Consider the function $f(x;a,b) = |x|^{p-4} \langle x, a \rangle \langle x, b \rangle$, which is $(p-2)$-homogeneous with respect to $x$. 
		It is not hard to see that  $f$ is continuous in $\mathbb{R}^{3N}$, and hence $f$ is uniformly continuous in any bounded subset of $\mathbb{R}^{3N}$.
		Componentwise, we have
		\begin{equation}
			(\Phi_n)_k 
			= 
			f(\nabla \varphi_1 + s t_n \nabla w_n; \nabla w_0, e_k) 
			\mbox{ and } 
			(\Phi)_k = f(\nabla \varphi_1; \nabla w_0, e_k),
			\quad k = 1,2,\dots,N,
		\end{equation}
		where $e_k$ is the unit $k$-th coordinate vector.
		Therefore, in view of \eqref{eq:unifconv} and the uniform continuity of $f$, we deduce that 
		$(\Phi_n)_k \to (\Phi)_k$ strongly in $L^\infty(\Omega\setminus E_{\varepsilon,k})$ for any $k = 1,2,\dots,N$, and hence 		
		$\Phi_n \to \Phi$ strongly in $L^\infty(\Omega\setminus E_{\varepsilon,k};\mathbb R^N)$. In particular, since $\Omega \setminus E_{\varepsilon,k}$ is bounded, we also have $\Phi_n \to \Phi$ strongly in $L^{\theta'}(\Omega\setminus E_{\varepsilon,k};\mathbb R^N)$.
		
		Observe that the function $a \mapsto \langle x,a\rangle^2$ is convex for any $x \in \mathbb{R}^N$, which yields the following inequality:
		\begin{equation}\label{eq:lower-bound-an}
			\langle x, a_n \rangle^2 \geqslant \langle x,a\rangle^2 + 2 \,\langle x,a\rangle \langle x,a_n-a\rangle
			\quad \text{for any}~ x, a_n, a \in \mathbb{R}^N.
		\end{equation}
		Substituting $x = \nabla \varphi_1 + s t_n \nabla w_n$, $a_n = \nabla w_n$,  and $a = \nabla w_0$ into \eqref{eq:lower-bound-an}, we obtain 
		\begin{align*}
			&\int_{\Omega \setminus E_{\varepsilon,k}} |\nabla \varphi_1 + s t_n \nabla w_n|^{p-4} 
			\langle \nabla \varphi_1 + s t_n \nabla w_n,\nabla w_n \rangle^2
			\, dx \\
			&\geqslant 
			\int_{\Omega \setminus E_{\varepsilon,k}} |\nabla \varphi_1 + s t_n \nabla w_n|^{p-4} 
			\langle \nabla \varphi_1 + s t_n \nabla w_n,\nabla w_0 \rangle^2 \, dx
			\\
			&+
			2
			\int_{\Omega \setminus E_{\varepsilon,k}} |\nabla \varphi_1 + s t_n \nabla w_n|^{p-4} 
			\langle \nabla \varphi_1 + s t_n \nabla w_n,\nabla w_0 \rangle \langle \nabla \varphi_1 + s t_n \nabla w_n,\nabla w_n - \nabla w_0 \rangle
			\, dx
			\\
			&= 
			\int_{\Omega \setminus E_{\varepsilon,k}} 
			\langle \Phi_n,\nabla w_0 \rangle  \, dx 
			+ 
			2\int_{\Omega \setminus E_{\varepsilon,k}} \langle \Phi_n,\nabla w_n - \nabla w_0 \rangle \, dx.
		\end{align*}
		From this point, we argue exactly as in the proof of \eqref{eq:convergence1} above (cf.\ \eqref{eq:sumtwoint1}) and establish the inequality \eqref{eq:convergence2}.
		We omit details.
		Combining now \eqref{eq:convergence1} and \eqref{eq:convergence2}, we finish the proof of the inequality \eqref{eq:convP1}.

		\smallskip
		Finally, recalling that $w_0 \in \mathcal{D}_{\varphi_1} \setminus \{0\}$ and $\intO \varphi_1^{q-1} w_0 \,dx = 0$, we use $w_0$ as an admissible function for the definition \eqref{eq:mu1} of $\mu_1$ and, taking into account \eqref{eq:convP0}, \eqref{eq:convP1}, and the convergence \eqref{eq:convP1P2toL}, we arrive at
		$$
		\mu_1 \leqslant \frac{\mathcal{P}_1(0,w_0)}{\mathcal{P}_0(0,w_0)}
		\leqslant
		\liminf_{n \to \infty} 
		\frac{\mathcal{P}_1(t_n,w_n)}{\mathcal{P}_0(t_n,w_n)}
		=
		\liminf_{n \to \infty} 
		\frac{\mathcal{P}_1(1,v_n)}{\mathcal{P}_0(1,v_n)} = \lambda_1 = \mu_1,
		$$
		where the last equality is given by Lemma \ref{lem:mu=lambda}. 
		Consequently, $w_0$ is a minimizer of $\mu_1$.
		However, according to Lemma \ref{lem:minimizers}, $w_0$ has to coincide with $\varphi_1$ up to a nonzero constant factor, which is impossible since $\intO \varphi_1^{q-1} w_0 \,dx = 0$.
		This contradiction finishes the proof of the claimed  inequality $\lambda_1 <\widetilde{\Lambda}_{\gamma_0}$ for a sufficiently small $\gamma_0>0$.
	\end{proof}
	
	\begin{remark}
		In the case $p=q>2$, the result of Proposition~\ref{prop:lambda<Lambda} is given by \cite[Lemma~5.2]{takac}.
		However, while the convergences \eqref{eq:convP0} and \eqref{eq:convP1} look intuitively correct, their rigorous justification appears to be delicate, cf.\  \cite[Proof of Lemma~5.2, p.~965]{takac}. 
		Our proof of \eqref{eq:convP1} relies on the continuous embedding $\D \hookrightarrow W_0^{1,\theta}(\Omega)$ provided by 
		Lemma~\ref{lem:embed}~\ref{lem:embed:2} (which is essentially due to \cite[Corollary 2.8]{brasco-lind} and \cite[Lemma~1.3, p.~238]{FNSS}, see a discussion in the proof of Lemma~\ref{lem:embed}~\ref{lem:embed:2} in Appendix \ref{sec:appendix} below) and inspired by a convexity argument from the proof of \cite[Theorem 1.1]{brasco-lind}.
		It is interesting to know if \eqref{eq:convP1} can be proved without this embedding result.
	\end{remark} 
	\begin{remark}
		Our proof of the equalities \eqref{eq:convP0} uses the assumption $q \geqslant 2$ in a principal way. 
		However, we anticipate that it is a technical assumption and \eqref{eq:convP0}, as well as Theorems~\ref{thm:main} and~\ref{thm:gen}, in general, hold true for any $q \in (1,p]$.
		Details are left for future investigation. 
	\end{remark}
	
	With the help of Proposition \ref{prop:lambda<Lambda}, we establish the result of Theorem \ref{thm:main} in the cone $\mathcal{C}_\gamma$ for  a sufficiently small $\gamma>0$.
	\begin{lemma}\label{lem:inside}
		Let $p \geqslant q \geqslant 2$ with $p>2$, and \ref{A} be satisfied.
		Let $\gamma_0>0$ be given by Proposition~\ref{prop:lambda<Lambda}.
		Then 
		there exists $C = C(\gamma_0,p,q,\Omega)>0$ such that 
		the improved Friedrichs inequality \eqref{eq:improvedFriedrichs} is satisfied for any $u \in \mathcal{C}_{\gamma_0}$.
	\end{lemma}
	\begin{proof}
		Take any $u \in \mathcal{C}_{\gamma_0}$ and recall that this assumption reads as
		$\|\nabla u^\perp\|_p \leqslant \gamma_0 |u^\parallel|$.
		If $u^\perp \equiv 0$ a.e.\ in $\Omega$ or $u^\parallel=0$, then there is nothing to prove. 
		Assume that $u^\perp \not\equiv 0$ a.e.\ in $\Omega$ and $u^\parallel \neq 0$.
		Consider the normalized function $\omega = u/u^\parallel$. 
		We have $\omega = \varphi_1 + \omega^\perp$, where
		$\omega^\perp = u^\perp/u^\parallel$ satisfies
		$0<\|\nabla \omega^\perp\|_p \leqslant \gamma_0$  and $\intO \varphi_1^{q-1} \omega^\perp \,dx = 0$.
		That is, $\omega^\perp$ is an admissible function for the definition \eqref{eq:tilde-lambda} of $\widetilde{\Lambda}_{\gamma_0}$.
		Using our notation, we write the following chain of equalities:
		\begin{align*}
			\frac{1}{p}\intO |\nabla (\varphi_1 + \omega^\perp)|^p \,dx 
			&- 
			\frac{\lambda_1}{p} \left(\intO|\varphi_1 + \omega^\perp|^q \,dx\right)^\frac{p}{q} 
			\\
			&=
			J[\varphi_1 + \omega^\perp] 
			= Q_{\omega^\perp}(\omega^\perp,\omega^\perp)
			= 
			\mathcal{P}_1(1,\omega^\perp) 
			- \lambda_1 
			\mathcal{P}_0(1,\omega^\perp).
		\end{align*}
		Proposition~\ref{prop:lambda<Lambda} guarantees that
		\begin{align*}
			\mathcal{P}_1(1,\omega^\perp) 
			- \lambda_1 
			\mathcal{P}_0(1,\omega^\perp)
			\geqslant
			\left(1-\frac{\lambda_1}{\widetilde{\Lambda}_{\gamma_0}}\right)\mathcal{P}_1(1,\omega^\perp)
			\geqslant
			\frac{C}{p}
			\left(
			\intO |\nabla \varphi_1|^{p-2} |\nabla \omega^\perp|^2 \,dx + \intO |\nabla \omega^\perp|^p \,dx
			\right),
		\end{align*}
		where the last inequality is given by  \eqref{eq:P1est}, $C>0$ does not depend on $\omega^\perp$, and the factor $1/p$ is chosen for convenience.
		Multiplying these two displayed expressions by $p|u^\parallel|^p$, we pass back to the function $u = u^\parallel \varphi_1 + u^\perp \in \mathcal{C}_{\gamma_0}$ and finally arrive at the claimed inequality:
		$$
		\intO |\nabla u|^p \,dx 
		- 
		\lambda_1 
		\left(\intO |u|^q \,dx\right)^\frac{p}{q}
		\geqslant 
		C\left(
		|u^\parallel|^{p-2}
		\intO |\nabla \varphi_1|^{p-2} |\nabla u^\perp|^2 \,dx + \intO |\nabla u^\perp|^p \,dx
		\right).
		\qedhere
		$$
	\end{proof}

	\subsection{Proof in \texorpdfstring{$\mathcal{C}_\gamma'$}{C-gamma'}}\label{sec:easy}
	
	Let us fix any $\gamma > 0$ and define the following critical value:
	\begin{equation}\label{eq:Lambdag}
		\Lambda_\gamma
		=
		\inf
		\left\{
		\frac{\intO |\nabla u|^p \,dx}{\left(\intO |u|^q \,dx\right)^\frac{p}{q}}:~
		u \in \mathcal{C}_\gamma' \setminus \{0\}
		\right\}.
	\end{equation}
	Clearly, we have $\|\nabla u^\perp\|_p > 0$ for any $u \in \mathcal{C}_\gamma' \setminus \{0\}$, since the latter reads as $\|\nabla u^\perp\|_p \geqslant \gamma |u^\parallel|$.
	Therefore, $\Lambda_\gamma$ can be equivalently characterized as
	\begin{equation}\label{eq:Lambdag2}
		\Lambda_\gamma
		=
		\inf
		\left\{
		\frac{\intO |t \nabla \varphi_1 + \nabla v|^p \,dx}{\left(\intO |t \varphi_1 + v|^q \,dx\right)^\frac{p}{q}}:~
		(t, v) \in \mathbb{R} \times \W,~
		|t| \leqslant \gamma^{-1}, \|\nabla v\|_p = 1, \intO \varphi_1^{q-1} v \,dx = 0 
		\right\}.
	\end{equation}
	It is evident from \eqref{eq:Lambdag} that $\lambda_1 \leqslant \Lambda_\gamma$. 
	The main property of $\Lambda_\gamma$ needed for the proof of Theorem~\ref{thm:main} in $\mathcal{C}_\gamma'$ is the fact that $\lambda_1 < \Lambda_\gamma$.
	We provide
	the following slightly more general result which is valid without assuming $p>2$, $q \geqslant 2$, and the regularity \ref{A} of $\Omega$.
	\begin{lemma}\label{lem:Cgamma-prime}
		Let $p \geqslant q>1$, $\gamma>0$, and $\Omega$ be a domain of finite measure.		
		Then $\lambda_1 < \Lambda_\gamma$.
	\end{lemma}
	\begin{proof}
		Suppose, by contradiction, that 
		$\lambda_1 = \Lambda_\gamma$.
		Then there exists a sequence $\{(t_n,v_n)\}$ which satisfies the constraints in the definition \eqref{eq:Lambdag2} of $\Lambda_\gamma$ and such that
		\begin{equation}\label{eq:prop1:conv1}
			\frac{\intO |t_n \nabla \varphi_1 + \nabla v_n|^p \,dx}{\left(\intO |t_n \varphi_1 + v_n|^q \,dx\right)^\frac{p}{q}}
			\to 
			\lambda_1
			\quad \text{as}~ n \to \infty.
		\end{equation}
		In view of the boundedness of $\{t_n\}$ and $\{\|\nabla v_n\|_p\}$, we deduce the existence of $t_0 \in [-\gamma^{-1},\gamma^{-1}]$ and $v_0 \in \W$ such that $t_n \to t_0$, and $v_n \to v_0$ weakly in $\W$ and strongly in $L^q(\Omega)$, along a common subsequence of indices.
		Thus, $t_n \varphi_1 + v_n 
		\to
		t_0 \varphi_1 + v_0$ weakly in $\W$ and strongly in $L^q(\Omega)$.
		Moreover, in view of the boundedness of $\varphi_1$ in $\Omega$, we have $\intO \varphi_1^{q-1} v_0 \, dx = 0$.
		Let us show that $t_0 \varphi_1 + v_0 \not\equiv 0$ a.e.\ in $\Omega$. 
		Suppose, by contradiction, that $v_0 = -t_0 \varphi_1$. 
		If $t_0 \neq 0$, then we get a contradiction to the ``orthogonality'' 
		$\intO \varphi_1^{q-1} v_0 \, dx = 0$ since $\varphi_1>0$ in $\Omega$.
		If $t_0 = 0$, then we deduce from \eqref{eq:prop1:conv1} and the strong convergence in $L^q(\Omega)$ that $\|\nabla v_n\|_p \to 0$, which contradicts the assumption $\|\nabla v_n\|_p=1$.
		Therefore, the function $t_0 \varphi_1 + v_0$ is nonzero.
		
		Thanks to the weak lower semicontinuity of the norm in $\W$,
		we get
		$$
		0<\intO |t_0 \nabla \varphi_1 + \nabla v_0|^p \,dx
		\leqslant
		\liminf_{n \to \infty}
		\intO |t_n \nabla \varphi_1 + \nabla v_n|^p \,dx,
		$$
		and hence the convergence \eqref{eq:prop1:conv1} and the  definition \eqref{eq:lambda1} of $\lambda_1$ yield $t_n \varphi_1 + v_n 
		\to t_0 \varphi_1 + v_0$ strongly in $\W$.
		In particular, we deduce that 
		$t_0 \varphi_1 + v_0$ is a minimizer of $\lambda_1$ and 
		$\|\nabla v_0\|_p =1$.
		Using the fact that the minimizer of $\lambda_1$ is unique modulo scaling, we obtain the existence of $s \in \mathbb{R} \setminus \{0\}$ such that $t_0 \varphi_1 + v_0 = s \varphi_1$.
		Since $v_0 \not\equiv 0$ a.e.\ in $\Omega$, we get $v_0 = (s-t_0) \varphi_1$ and $s \neq t_0$.
		However, this again contradicts the ``orthogonality'' 
		$\intO \varphi_1^{q-1} v_0 \, dx = 0$ since $\varphi_1>0$ in $\Omega$.
	\end{proof}
	
	\begin{remark}
		It is evident from the proof that Lemma~\ref{lem:Cgamma-prime} remains valid without the assumption $p \geqslant q$ provided that $1<q<p^*$ and $\varphi_1$ is a unique (modulo scaling) minimizer of $\lambda_1$.
		In the case $p=q>1$, Lemma~\ref{lem:Cgamma-prime} is given by \cite[Lemma~5.1]{takac}.
	\end{remark}

	Using Lemma \ref{lem:Cgamma-prime}, we establish the result of Theorem \ref{thm:main} in the cone $\mathcal{C}_\gamma'$ for any $\gamma>0$ and observe that it is valid for any $q \in (1,p]$.
	\begin{lemma}\label{lem:outside}
		Let $p \geqslant q > 1$ with $p>2$, and \ref{A} be satisfied.
		Then for any $\gamma>0$ there exists $C=C(\gamma,p,q,\Omega)>0$ such that the improved Friedrichs inequality \eqref{eq:improvedFriedrichs} is satisfied for any $u \in \mathcal{C}_\gamma'$.
	\end{lemma}
	\begin{proof}
		Let $u \in \mathcal{C}_\gamma'$. 
		If $u \equiv 0$ a.e.\ in $\Omega$, then there is nothing to prove.
		Assuming that $u \in \mathcal{C}_\gamma' \setminus \{0\}$, we see that $u$ is an admissible function for the definition \eqref{eq:Lambdag} of $\Lambda_\gamma$, which yelds
		\begin{equation}\label{eq:lem:Cprime-1}
			\intO |\nabla u|^p \,dx
			-
			\lambda_1 \left(\intO |u|^q \,dx\right)^\frac{p}{q}
			\geqslant 
			\left(1 - \frac{\lambda_1}{\Lambda_\gamma}\right)
			\intO |\nabla u|^p \,dx.
		\end{equation}
		According to Lemma~\ref{lem:Cgamma-prime}, we have $\lambda_1< \Lambda_\gamma$. 
		We want to prove that the right-hand side of \eqref{eq:lem:Cprime-1} can be estimated from below in the same form as the right-hand side of \eqref{eq:improvedFriedrichs}.
		If $u^\parallel = 0$, then this fact is evident. 
		Assume that $u^\parallel \neq 0$.
		Dividing \eqref{eq:lem:Cprime-1} by $|u^\parallel|^p$, we get
		\begin{equation}\label{eq:lem:Cprime-2}
			\intO |\nabla (\varphi_1 + \omega^\perp)|^p \,dx
			-
			\lambda_1 \left(\intO |\varphi_1 + \omega^\perp|^q \,dx\right)^\frac{p}{q}
			\geqslant 
			\left(1 - \frac{\lambda_1}{\Lambda_\gamma}\right)
			\intO |\nabla (\varphi_1 + \omega^\perp)|^p \,dx,
		\end{equation}
		where $\omega^\perp = u^\perp/u^\parallel$ satisfies $\|\nabla \omega^\perp\|_p \geqslant \gamma$ and $\intO \varphi_1^{q-1} \omega^\perp \,dx = 0$.
		First, we show the existence of $C>0$ independent of $\omega^\perp$ such that
		\begin{equation}\label{eq:lowernwew}
			\intO |\nabla (\varphi_1 + \omega^\perp)|^p \,dx 
			\geqslant C 
			\intO |\nabla \omega^\perp|^p \,dx.
		\end{equation}
		Suppose, contrary to our claim, that there exists a sequence $\{\omega^\perp_n\} \subset \W$ satisfying $\|\nabla \omega^\perp_n\|_p \geqslant \gamma$, $\intO \varphi_1^{q-1} \omega^\perp_n \,dx = 0$, and
		\begin{equation}\label{eq:contrlem}
			\intO |\nabla (\varphi_1 + \omega^\perp_n)|^p \,dx 
			\leqslant \frac{1}{n}
			\intO |\nabla \omega^\perp_n|^p \,dx.
		\end{equation}
		It is not hard to see that $\{\|\nabla \omega^\perp_n\|_p\}$ is bounded, and hence there exists $\omega^\perp_0 \in \W$ such that $\omega^\perp_n \to \omega^\perp_0$ weakly in $\W$ and strongly in $L^q(\Omega)$, up to a subsequence. 
		In particular, we have $\intO \varphi_1^{q-1} \omega^\perp_0 \,dx = 0$. 
		We deduce from \eqref{eq:contrlem} that 
		$$
		0 \leqslant 
		\intO |\nabla (\varphi_1 + \omega^\perp_0)|^p \,dx 
		\leqslant
		\liminf_{n \to \infty} \intO |\nabla (\varphi_1 + \omega^\perp_n)|^p \,dx 
		\leqslant
		\liminf_{n \to \infty}
		\frac{1}{n}
		\intO |\nabla \omega^\perp_n|^p \,dx
		= 0.
		$$
		Consequently, $\intO |\nabla (\varphi_1 + \omega^\perp_0)|^p \,dx=0$, which leads to $\omega^\perp_0 = -\varphi_1$. 
		However, this is impossible since $\intO \varphi_1^{q-1} \omega^\perp_0 \,dx = 0$.
		This proves the lower bound \eqref{eq:lowernwew}. 
		
		Second, we show the existence of $C>0$ independent of $\omega^\perp$ such that
		\begin{equation}\label{eq:lowernwewD}
			\intO |\nabla \omega^\perp|^p \,dx
			\geqslant
			C
			\intO |\nabla \varphi_1|^{p-2} |\nabla \omega^\perp|^2 \,dx.
		\end{equation}
		We see that if $\|\omega^\perp\|_{\varphi_1} \leqslant \gamma^{p/2}$, then \eqref{eq:lowernwewD} is satisfied with $C=1$ since $\|\nabla \omega^\perp_n\|_p \geqslant \gamma$ by the choice of $u$.
		On the other hand, if $\|\omega^\perp\|_{\varphi_1} \geqslant \gamma^{p/2}$, then the continuous embedding $\W \hookrightarrow \D$ (see Lemma~\ref{lem:embed}~\ref{lem:embed:3}) gives
		\begin{equation*}\label{eq:lowernwewD2}
			\intO |\nabla \omega^\perp|^p \,dx
			\geqslant
			C
			\left(\intO |\nabla \varphi_1|^{p-2} |\nabla \omega^\perp|^2 \,dx
			\right)^\frac{p}{2}
			\geqslant 
			C
			\gamma^\frac{p(p-2)}{2}
			\intO |\nabla \varphi_1|^{p-2} |\nabla \omega^\perp|^2 \,dx.
		\end{equation*}	
		Thus, the estimate \eqref{eq:lowernwewD} holds true.
		
		Combining \eqref{eq:lem:Cprime-2} with the lower bounds \eqref{eq:lowernwew} and \eqref{eq:lowernwewD}, and multiplying the resulting expression by $|u^\parallel|^p$, we pass back to the function $u = u^\parallel \varphi_1 + u^\perp \in \mathcal{C}_\gamma'$ and deduce the desired inequality
		$$
		\intO |\nabla u|^p \,dx 
		- 
		\lambda_1 
		\left(\intO |u|^q \,dx\right)^\frac{p}{q}
		\geqslant 
		C\left(
		|u^\parallel|^{p-2}
		\intO |\nabla \varphi_1|^{p-2} |\nabla u^\perp|^2 \,dx + \intO |\nabla u^\perp|^p \,dx
		\right).
		\qedhere
		$$
	\end{proof}

	\smallskip
	\noindent
	The proof of Theorem \ref{thm:main} follows by combining Lemmas \ref{lem:inside} and \ref{lem:outside}.

	\section{Generalization. Proof of Theorem \ref{thm:gen}}\label{sec:general:proof}

	Let us denote the expression (without the constant $C$) on the right-hand side of the improved Friedrichs inequality~\eqref{eq:improvedFriedrichs-gen} as $M_l[u]$, i.e.,
	\begin{equation}\label{eq:Ml}
		M_l[u] 
		= 
		\big| l[u]\big |^{p-2} \|Pu\|_{\varphi_1}^2
		+
		\|\nabla Pu\|_p^p,
	\end{equation}
	where $Pu = u - l[u]\varphi_1$ is the projection to $\text{Ker}(l)$, and $l[\varphi_1]=1$.
	
	Our goal is to prove the following equivalence statement.
	\begin{proposition}\label{prop:equiv}
		Let $p \geqslant 2$ and \ref{A} be satisfied.
		Let $l_1, l_2: \W \to \mathbb{R}$ be two distinct bounded linear functionals such that $l_1[\varphi_1] = l_2[\varphi_1] = 1$.
		Then there exist constants $C_1, C_2>0$ such that $C_1 M_{l_1}[u] \leqslant M_{l_2}[u] \leqslant C_2 M_{l_1}[u]$ for any $u \in W_0^{1,p}(\Omega)$.
	\end{proposition}
	
	Once Proposition~\ref{prop:equiv} is established, Theorem~\ref{thm:gen} follows directly from Theorem~\ref{thm:main} by taking one of the functionals, say, $l_1$, as $l_1[u]=\intO \varphi_1^{q-1} u \,dx/\intO \varphi_1^{q}\,dx$. 
	(In fact, in order to prove Theorem~\ref{thm:main}, it is sufficient to apply only one part of Proposition~\ref{prop:equiv}, i.e., that $M_{l_2}[u] \leqslant C_2 M_{l_1}[u]$.)
	
	\begin{proof}[Proof of Proposition~\ref{prop:equiv}]
		We start with several auxiliary observations.
		Since $l_1$ and $l_2$ are distinct and not linearly dependent (in view of the normalization assumption $l_1[\varphi_1] = l_2[\varphi_1] = 1$), their kernels do not coincide, and hence there exists $v_1 \in W_0^{1,p}(\Omega)$ such that $l_1[v_1] = 1$ and $l_2[v_1] = 0$.  
		Then we define $v_2 = \varphi_1 - v_1$, so that $l_1[v_2] = 0$ and $l_2[v_2] = 1$. 
		Finally, we introduce $\psi = v_1 - v_2 = 2 v_1 - \varphi_1$, and hence $l_1[\psi] = 1$ and $l_2[\psi] = -1$. 
		Notice that $\psi \not\in \mathbb{R}\varphi_1$.
		With the help of $\varphi_1$ and $\psi$, any $u \in \W$ can be decomposed as
		\begin{equation}
			\label{even-odd_decomposition}
			u = \alpha \varphi_1 + \beta \psi + w, 
		\end{equation}
		where 
		$$
		\alpha = \frac{l_1[u]+l_2[u]}{2},
		\quad
		\beta = \frac{l_1[u]-l_2[u]}{2},
		$$
		and $w \in \text{Ker}(l_1) \cap \text{Ker}(l_2)$, i.e., $l_1[w] = l_2[w] = 0$. 
		We have
		\begin{align}
			\label{eq:l1p1}
			l_1[u] &= \alpha + \beta, 
			\quad 
			P_1 u = \beta (\psi - \varphi_1) + w,\\
			\label{eq:l2p2}
			l_2[u] &= \alpha - \beta,
			\quad 
			P_2 u = \beta (\psi + \varphi_1) + w.
		\end{align}
		Using \eqref{eq:l1p1} and \eqref{eq:l2p2}, we rewrite $M_{l_1}$ and $M_{l_2}$ defined by \eqref{eq:Ml} as
		\begin{align*}
			M_{l_1}[u]
			&=
			|\alpha+\beta|^{p-2}
			\|\beta (\psi -\varphi_1) + w\|_{\varphi_1}^2
			+
			\|\beta \,\nabla(\psi - \varphi_1) + \nabla w\|_p^p,
			\\
			M_{l_2}[u]
			&=	
			|\alpha-\beta|^{p-2}
			\|\beta (\psi +\varphi_1) + w\|_{\varphi_1}^2
			+
			\|\beta \,\nabla(\psi + \varphi_1) + \nabla w\|_p^p.
		\end{align*}
		Since $\varphi_1$ and $\psi$ are fixed, we employ the triangle inequality to
		estimate $M_{l_1}$ from above as follows:
		\begin{align}
			\label{first_seminorm_estimate}
			M_{l_1}[u]
			\leq 
			C_1 
			|\alpha + \beta|^{p-2} (|\beta| + \|w\|_{\varphi_1})^2
			+
			C_1
			(|\beta| + \|\nabla w\|_p)^p,
		\end{align}
		where $C_1>0$ does not depend on $u$. 
		A similar estimate holds true also  for $M_{l_2}$.
		On the other hand, 
		recalling that $\W \hookrightarrow \D$, $l_1[\psi+\varphi_1]=2$, and $w \in \text{Ker}(l_1)$, we apply Lemma~\ref{lem:trian} (see below) with $l=l_1$ to provide the following lower estimate for $M_{l_2}$:
		\begin{align}
			\label{second_seminorm_estimate2}
			M_{l_2}[u] 
			\geq
			C_2 
			|\alpha - \beta|^{p-2} (|\beta| + \|w\|_{\varphi_1})^2
			+
			C_2
			(|\beta| + \|\nabla w\|_p)^p,
		\end{align}
		where $C_2>0$ is independent of $u$.
		A similar estimate can be derived also for $M_{l_1}$.
		
		Assume now, contrary to the claim of the proposition, that $M_{l_1}$ and $M_{l_2}$ are not equivalent.
		That is, noting that $M_{l_1}$ and $M_{l_2}$ are $p$-homogeneous, 
		we can assume, without loss of generality, that there exists a sequence $\{u_n\} \subset \W$ such that $M_{l_1}[u_n] = 1$ and $M_{l_2}[u_n] \to 0$. 
		Using the decomposition \eqref{even-odd_decomposition}, we write 
		\begin{equation}
			u_n = \alpha_n \varphi_1 + \beta_n \psi + w_n.    
		\end{equation}
		We obtain from \eqref{second_seminorm_estimate2} that $\beta_n \to 0$ and $w_n \to 0$ strongly in $W_0^{1,p}(\Omega)$. 
		If $p=2$, then we get a contradiction to \eqref{first_seminorm_estimate} and $M_{l_1}[u_n] = 1$.
		Hence, assume that $p>2$.
		The convergence $w_n \to 0$ strongly in $W_0^{1,p}(\Omega)$ implies that $w_n \to 0$ strongly in $\D$, see  Lemma~\ref{lem:embed}~\ref{lem:embed:3}.
		We deduce from \eqref{first_seminorm_estimate} and $M_{l_1}[u_n] = 1$ that
		\begin{equation}
			\lim_{n\to \infty} |\alpha_n + \beta_n|^{p-2} (|\beta_n| + \|w_n\|_{\varphi_1})^2 
			\geq C_1^{-1},
		\end{equation}
		which yields $\alpha_n \to \infty$. But then
		\begin{equation}
			\lim_{n\to \infty} |\alpha_n - \beta_n|^{p-2} (|\beta_n| + \|w_n\|_{\varphi_1})^2 
			= 
			\lim_{n \to \infty} \frac {|\alpha_n - \beta_n|^{p-2}}{|\alpha_n + \beta_n|^{p-2}} \lim_{n\to \infty} |\alpha_n + \beta_n|^{p-2} (|\beta_n| + \|w_n\|_{\varphi_1})^2 \geq C_1^{-1},  
		\end{equation}
		which contradicts \eqref{second_seminorm_estimate2} in view of the assumption $M_{l_2}[u_n] \to 0$.
	\end{proof}
	
	Let us provide the following auxiliary result about an inverse triangle type inequality which is used in the proof of Proposition~\ref{prop:equiv}.
	\begin{lemma}\label{lem:trian}
		Let $X$ be a Banach space. 
		Let $l: X \to \mathbb{R}$ be a bounded linear functional. 
		Let $\omega \in X$ be such that $l[\omega] \ne 0$. Then there exists a constant $C>0$ such that for any $u \in \mathbb{R} \omega$ and $v \in \mathrm{Ker} (l)$ the following estimate holds:
		\begin{equation}\label{eq:inverse}
			\|u+v\| \geqslant C(\|u\|+\|v\|).   
		\end{equation}
	\end{lemma}
	
	\begin{proof}
		Assume, without loss of generality, that $l[\omega] = 1$. 
		Let us take any $u \in \mathbb{R}\omega$ and $v \in \mathrm{Ker} (l)$, and 
		define $\phi = u+v$.
		Since $X = \mathbb{R}\omega \oplus  \mathrm{Ker} (l)$, we have 
		$\phi = l[\phi] \omega + P \phi$ and hence, by the uniqueness of the decomposition,
		$u=l[\phi]\omega$ and 
		$v=P \phi$, where $P: X \to \mathrm{Ker} (l)$ is a projection operator defined as $P \phi = \phi - l[\phi] \omega$.
		Then
		\begin{equation}
			\|u\|+\|v\| 
			= 
			|l[\phi]|\,\|\omega\| + \|P \phi\| 
			\leqslant 
			2|l[\phi]|\,\|\omega\| + \|\phi\| 
			\leqslant
			(2\|l\|_*\|\omega\| + 1) \|\phi\|,   
		\end{equation}
		where $\|\cdot\|_*$ is the operator norm.
		Denoting $C = (2\|l\|_*\|\omega\| + 1)^{-1}$, we obtain \eqref{eq:inverse}.
	\end{proof}

	\section{Application. Proof of Theorem \ref{thm:existence}}\label{sec:application}
	
	The $C^1(\W; \mathbb{R})$-energy functional associated with the problem \eqref{eq:P} is given by
	$$
	E[u] 
	= 
	\frac{1}{p} \intO |\nabla u|^p \,dx
	-
	\frac{\lambda_1}{p}
	\left(\intO |u|^q \,dx\right)^\frac{p}{q} 
	- f[u].
	$$
	Let us show that under the assumptions of Theorem~\ref{thm:existence} the functional $E$
	is bounded from below and satisfies $E[u] \geqslant o(1)$ as $\|\nabla u\|_p \to \infty$. 
	This will imply that $E$ possesses a global minimizer.
	Throughout the proof, we denote by $C>0$ a universal constant.
	
	We decompose any $u \in \W$ as in \eqref{eq:decomp}, i.e., $u = u^\parallel \varphi_1 + u^\perp$, where $u^\parallel$ and $u^\perp$ are defined by \eqref{eq:decomp2}.
	By Lemma \ref{lem:embed} \ref{lem:embed:0}, we have $\|\nabla u^\perp\|_p \geqslant C \|u^\perp\|_{\varphi_1}$ and
	$$
	f[u]
	=
	f[u^\perp]
	\leqslant
	\|f\|_* \|u^\perp\|_{\varphi_1}
	\leqslant
	C\|f\|_* \|\nabla u^\perp\|_p,
	$$
	where $\|\cdot\|_*$ stands for the operator norm.
	Applying Theorem \ref{thm:main} (see \eqref{eq:improvedFriedrichs-norm}), we deduce that
	\begin{equation}
		\label{eq:E1}
		E[u] 
		\geqslant 
		C\left(
		|u^\parallel|^{p-2} \|u^\perp\|_{\varphi_1}^2  + \|\nabla u^\perp\|_p^p
		\right)
		-
		C\|f\|_* \|\nabla u^\perp\|_p
	\end{equation}
	and
	\begin{equation}
		\label{eq:E2}
		E[u]
		\geqslant
		C
		|u^\parallel|^{p-2} \|u^\perp\|_{\varphi_1}^2 
		+ 
		C
		\|u^\perp\|_{\varphi_1}^p
		-
		\|f\|_* \|u^\perp\|_{\varphi_1}.
	\end{equation}
	We see from \eqref{eq:E1} (or \eqref{eq:E2}) that $E$ is bounded from below. 
	If $\|\nabla u\|_p \to \infty$, then we have $|u^\parallel| \to \infty$ or $\|\nabla u^\perp\|_p \to \infty$.
	In the latter case, \eqref{eq:E1} yields $E[u] \to +\infty$.
	If $|u^\parallel| \to \infty$ and $\{\|u^\perp\|_{\varphi_1}\}$ is separated from zero, then $E[u] \to +\infty$ as well, as it follows from \eqref{eq:E2}.
	Finally, if $|u^\parallel| \to \infty$ and $\|u^\perp\|_{\varphi_1} \to 0$, then we deduce from \eqref{eq:E2} that $E[u] \geqslant -\|f\|_2 \|u^\perp\|_{\varphi_1} = o(1)$.
	On the other hand, since $p,q>1$ and $f$ is nonzero, it is not hard to observe that $\inf_{\W} E < 0$.
	Thus, any minimizing sequence for $E$ is bounded in $\W$ and hence, by a standard argument, it converges strongly in $\W$ to a global minimizer of $E$ which is a (weak) solution of \eqref{eq:P}, up to a subsequence.
	\qed

	\section{Alternative improvement. Proof of Theorem \ref{thm:Br}}\label{sec:Br}	
	
	First, let us take any $w \in \W \setminus \{0\}$ such that 
	\begin{equation}\label{eq:1}
		\intO |w|^q \,dx = \intO \varphi_1^q \, dx.
	\end{equation}
	By the definition \eqref{eq:lambda1} of $\lambda_1$, we have
	\begin{equation}\label{eq:2}
		\intO |\nabla w|^p \,dx \geqslant \intO |\nabla \varphi_1|^p \, dx.
	\end{equation}
	Set
	$$
	\sigma_t = ((1-t) |w|^p + t \varphi_1^p)^\frac{1}{p},
	\quad
	t \in [0,1].
	$$
	By the enhanced hidden convexity \cite[Eq.~(2.10)]{BPZ} (see also  \cite[Eq.~(2), p.~178]{LL} for the case $p=2$ and $t=1/2$), we have
	\begin{equation}\label{eq:Pimp-proof1}
		\intO |\nabla \sigma_t|^p \,dx 
		+  
		Ct(1-t) \intO \mathcal{R}_p(|w|, \varphi_1;t) \,dx 
		\leqslant 
		(1-t) \intO |\nabla w|^p \,dx
		+
		t \intO |\nabla \varphi_1|^p \,dx \leqslant 
		\intO |\nabla w|^p \,dx,
	\end{equation}
	where the last inequality follows from \eqref{eq:2} and $\mathcal{R}_p$ is given by \eqref{eq:R}.
	In particular, $\sigma_t \in \W$ for all $t \in [0,1]$.
	On the other hand, by the concavity of the map $s \mapsto s^{q/p}$ and the equality \eqref{eq:1}, we get
	\begin{equation}\label{eq:Pimp-proof2}
		\intO \sigma_t^q \,dx 
		\geqslant 
		(1-t) \intO |w|^q \,dx
		+
		t \intO \varphi_1^q \,dx 
		=
		\intO |w|^q \,dx.
	\end{equation}
	Raising both sides of \eqref{eq:Pimp-proof2} to the power $p/q$, multiplying  the result by $\lambda_1$ and subtracting from \eqref{eq:Pimp-proof1}, we obtain
	$$
	\intO |\nabla w|^p \,dx
	-
	\lambda_1 \left(\intO |w|^q \,dx\right)^\frac{p}{q}
	\geqslant
	Ct(1-t) \intO \mathcal{R}_p(|w|,\varphi_1;t) \,dx
	+
	\intO |\nabla \sigma_t|^p \,dx 
	-
	\lambda_1 \left(\intO \sigma_t^q \,dx\right)^\frac{p}{q}
	$$
	for any $t \in [0,1]$.
	Using \eqref{eq:Friedrichs}, we estimate the difference of integrals containing $\sigma_t$ by zero and hence derive
	\begin{equation}\label{eq:Pimpr0}
		\intO |\nabla w|^p \,dx
		-
		\lambda_1 \left(\intO |w|^q \,dx\right)^\frac{p}{q}
		\geqslant
		C \max_{t \in [0,1]}
		\left[
		t(1-t) \intO \mathcal{R}_p(|w|,\varphi_1;t) \,dx
		\right]
	\end{equation}
	for any $w \in \W$ satisfying the normalization \eqref{eq:1}.
	
	Now we take \textit{any} $u \in \W \setminus \{0\}$.
	Consider $w = s u$, where $s = \|\varphi_1\|_q/\|u\|_q$.
	Such $w$ satisfies \eqref{eq:1} and hence \eqref{eq:Pimpr0}.
	We get
	$$
	s^p \left(\intO |\nabla u|^p \,dx
	-
	\lambda_1 \left(\intO |u|^q \,dx\right)^\frac{p}{q}
	\right)
	\geqslant
	C \max_{t \in [0,1]}
	\left[
	t(1-t) \intO \mathcal{R}_p(s|u|,\varphi_1;t) \,dx
	\right],
	$$
	that is,
	\begin{equation*}\label{eq:Pimpr2}
		\intO |\nabla u|^p \,dx
		-
		\lambda_1 \left(\intO |u|^q \,dx\right)^\frac{p}{q}
		\geqslant
		\frac{C\|u\|_q^p}{\|\varphi_1\|_q^p} \, 
		\max_{t \in [0,1]}
		\left[
		t(1-t) \intO \mathcal{R}_p\left(\frac{\|\varphi_1\|_q}{\|u\|_q}|u|,\varphi_1;t\right)dx
		\right],
	\end{equation*}
	which is exactly the improved Friedrichs inequality \eqref{eq:Pimpr}.
	
	\smallskip
	Let us now justify the improved Poincar\'e inequality \eqref{eq:improvedPoinc2}. 
	Take any $u \in \W \setminus \{0\}$ and set, as above,
	$$
	\widetilde{\sigma}_t = ((1-t) |u|^p + t \varphi_1^p)^\frac{1}{p},
	\quad
	t \in [0,1].
	$$
	By the enhanced hidden convexity, we have
	\begin{equation}\label{eq:Pimp-proof1-1}
		\intO |\nabla \widetilde{\sigma}_t|^p \,dx 
		+  
		Ct(1-t) \intO \mathcal{R}_p(|u|, \varphi_1;t) \,dx 
		\leqslant 
		(1-t) \intO |\nabla u|^p \,dx
		+
		t \intO |\nabla \varphi_1|^p \,dx,
	\end{equation}
	and
	\begin{equation}\label{eq:Pimp-proof2-1}
		\lambda_1 \intO \widetilde{\sigma}_t^p \,dx 
		=
		\lambda_1 (1-t) \intO |u|^p \,dx
		+
		\lambda_1 t \intO \varphi_1^p \,dx.
	\end{equation}
	Subtracting \eqref{eq:Pimp-proof2-1} from \eqref{eq:Pimp-proof1-1}, recalling that $\varphi_1$ is the minimizer of $\lambda_1$, 
	and using \eqref{eq:Friedrichs} to estimate the difference of integrals containing $\widetilde{\sigma}_t$ by zero, 
	we get
	$$
	\intO |\nabla u|^p \,dx
	-
	\lambda_1 \intO |u|^p \,dx
	\geqslant
	Ct \intO \mathcal{R}_p(|u|,\varphi_1;t) \,dx
	$$
	for any $t \in [0,1)$.
	Tending now $t$ to $1$, we derive 
	the desired inequality \eqref{eq:improvedPoinc2}.
	Finally, if $p \geqslant 2$, then the definition \eqref{eq:R} of $\mathcal{R}_p$ gives
	\begin{equation}\label{eq:imporved-step}
		\intO |\nabla u|^p \,dx 
		- 
		\lambda_1 
		\intO |u|^p \,dx
		\geqslant
		C \intO \left|\nabla |u| - \frac{|u|}{\varphi_1} \nabla \varphi_1 \right|^p \,dx.
	\end{equation}
	Observing that
	$$
	\left|\nabla |u| - \frac{|u|}{\varphi_1} \nabla \varphi_1 \right|
	=
	\left|\nabla u - \frac{u}{\varphi_1} \nabla \varphi_1 \right|
	=
	\left|\nabla \left(\frac{u}{\varphi_1}\right)\right| \varphi_1
	\quad \text{a.e.\ in}~ \Omega,
	$$
	we obtain the inequality  \eqref{eq:improvedPoinc3}.
	\qed

	\appendix
	\section{Appendix. Proof of Lemma~\ref{lem:embed}}\label{sec:appendix}

	We recall that $p>2$, $1<q<p^*$, and \ref{A} is satisfied.
	For convenience, throughout the proof, we  sometimes denote the norms in $W_0^{1,r}(\Omega)$ and $L^r(\Omega)$ (for $r \geqslant 1$) as $\|\cdot\|_{W_0^{1,r}(\Omega)}$ and $\|\cdot\|_{L^r(\Omega)}$, respectively, to reflect the dependence on $\Omega$. 
	Moreover, we denote by $C>0$ a universal constant whose value may vary from  inequality to inequality.
	
	\ref{lem:embed:0}
	Let us show the existence of $\kappa \in (2,2^*)$ such that the embedding $\D \hookrightarrow L^\kappa(\Omega)$ is compact.
	We start with several remarks.
	In the case $p=q>2$, the compactness of $\D \hookrightarrow L^2(\Omega)$ is proved in \cite[Lemma~4.2]{takac-only}. Although \cite[Lemma~4.2]{takac-only} requires $\Omega$ to satisfy the assumption \ref{A} \textit{and} the interior sphere condition when $N \geqslant 2$, the latter requirement is used only to guarantee that the first eigenfunction $\varphi_1$ obeys the Hopf maximum principle on $\partial \Omega$ (see \cite[(2.2)]{takac-only}), which is by now known to be true assuming \ref{A} alone, see \cite{melkshah}. 
	The arguments of \cite[Lemma~4.2]{takac-only} can be slightly amended to cover the general case $p > 2$ and $1<q<p^*$.
	In the case $N=1$ and when $p=q>2$, the compactness of the embedding  $\D \hookrightarrow L^\kappa(\Omega)$ for any $\kappa>1$ follows from \cite[Lemma~1.3, p.~238]{FNSS}, and analogous arguments can be applied to cover any $p>2$ and $q>1$.
	Thus, hereinafter, we will be interested only in the case $N \geqslant 2$.
	
	The continuity of the embedding  $\D \hookrightarrow L^{\sigma}(\Omega)$ for some $\sigma>2$ follows from \cite[Theorem~3.1, (3.6)]{sciunz} in combination with \cite[Theorem~2.3]{sciunz}. 
	The results of \cite{sciunz} are formulated under an abstract smoothness assumption on $\Omega$. 
	Nevertheless, the assumption \ref{A} is sufficient for our aims, which is shown in \cite[Theorem~2.7]{brasco-lind}. 
	In the current form,  \cite[Theorem~2.7]{brasco-lind} requires $q \geqslant p$, but the proof remains valid with no changes at least for a \textit{fixed} $q \in (1,p^*)$.
	Indeed, the assumption $q \geqslant p$ is used in the proof of \cite[Theorem~2.7]{brasco-lind} only in the derivation of an explicit and uniform $L^\infty(\Omega)$-bound for minimizers of $\lambda_1$ with respect to $q$, see \cite[Proposition~2.4]{brasco-lind}.
	Since we are interested in a \textit{fixed} $q \in (1,p^*)$, it is sufficient to substitute \cite[Proposition~2.4]{brasco-lind} 
	with the fact that any minimizer of $\lambda_1$ belongs to $L^\infty(\Omega)$, see, e.g., \cite[Theorem~II]{otani}.
	
	Based on the continuity of the embedding $\D \hookrightarrow L^{\sigma}(\Omega)$, we show that $\D \hookrightarrow L^{\kappa}(\Omega)$  compactly for some $\kappa  \in (2,\sigma)$. 
	Our proof is inspired by \cite[Lemma~4.2]{takac-only}, but we provide more details in subtle places.
	We assume, without loss of generality, that $\sigma \in (2,2^*)$.
	
	Let $\Omega_\delta$ be a strip of width $\delta>0$ around the boundary $\partial \Omega$:
	\begin{equation}
		\Omega_\delta = \lbrace x \in \Omega:~ \mathrm{dist}(x,\partial \Omega) < \delta \rbrace.
	\end{equation}
	Under the assumption \ref{A}, we have $\varphi_1 \in C^{1}(\overline{\Omega})$ (see \cite{Lieberman}) and it follows from \cite{melkshah} that $\varphi_1$ satisfies $\partial \varphi_1/ \partial \nu < 0$ on $\partial\Omega$, where $\nu$ is the outward unit normal vector to $\partial \Omega$.
	Thus, for any sufficiently small $\delta>0$ there exists $C>0$ such that $|\nabla \varphi_1| \geqslant C > 0$ in $\Omega_\delta$. 
	As a consequence, we can find $C>0$ such that
	\begin{equation}\label{eq:equivalentnorms}
		C^{-1} \|v\|_{W_0^{1,2}(\Omega_\delta)} \leqslant \|v\|_{\varphi_1} \leqslant C \|v\|_{W_0^{1,2}(\Omega_\delta)}
		\quad \text{for any}~~ v \in C_0^\infty(\Omega_\delta).
	\end{equation}		
	Set $\kappa = \sigma/2+1$ and observe that $\kappa \in (2,\sigma)$.
	Taking any $v \in C_0^\infty(\Omega)$, we use $|v|^\kappa$ as a test function for \eqref{eq:varphi-solution} and get
	\begin{align}
		\notag
		&\lambda_1 \left(\intO \varphi_1^q \,dx\right)^\frac{p-q}{q} 
		\intO \varphi_1^{q-1} |v|^\kappa \,dx
		=
		\intO |\nabla \varphi_1|^{p-2} \langle \nabla \varphi_1, \nabla (|v|^\kappa) \rangle \,dx
		\\
		\label{eq:47takac}
		&\leqslant
		\kappa 
		\intO |\nabla \varphi_1|^{\frac{p}{2}-1}  |\nabla v| \,
		|\nabla \varphi_1|^{\frac{p}{2}} |v|^{\kappa-1} \,dx
		\leqslant
		\kappa \|v\|_{\varphi_1} \left(
		\intO |\nabla \varphi_1|^{p} |v|^{\sigma} \,dx
		\right)^{\frac{1}{2}}.
	\end{align}	
	Fix a cut-off function $\xi \in C_0^\infty(\Omega)$ such that
	\begin{equation*}
		\xi(x) = 0 ~\mbox{ if } x \in \Omega_\delta
		\quad \text{and} \quad 
		\xi(x) = 1 ~\mbox{ if } x \in \Omega \setminus \Omega_{2\delta},
	\end{equation*}
	and consider the decomposition $v = \xi v + (1-\xi) v$.
	We use the continuity of the embedding $\D \hookrightarrow L^2(\Omega)$ and the $C^1(\overline{\Omega})$-regularity of $\varphi_1$ to get
	\begin{align*}
		\|(1 - \xi) v\|_{\varphi_1}^2 
		=
		\intO |\nabla \varphi_1|^{p-2}((1 - \xi) \nabla v - \nabla \xi \cdot v)^2 \,dx
		\leqslant
		C\|v\|_{\varphi_1}^2
		+
		C\|v\|_{L^2(\Omega)}^2
		\leqslant
		C\|v\|_{\varphi_1}^2,
	\end{align*}
	where $C>0$ does not depend on $v \in C_0^\infty(\Omega)$.
	Thus, by the density of $C_0^\infty(\Omega)$ in $\D$, we conclude that the multiplication by $(1-\xi)$ is a bounded linear operator in $\D$.
	Consequently, the multiplication by $\xi$ is also a bounded linear operator in $\D$.
	
	Thanks to the continuity of the embedding $\D \hookrightarrow L^{\sigma}(\Omega)$, any sequence $\{v_n\} \subset \D$ which converges weakly in $\D$ must also converge weakly in $L^{\sigma}(\Omega)$ and $L^{\kappa}(\Omega)$ to the same limit.
	Assume, without loss of generality, that this limit is zero and that $\{v_n\} \subset C_0^\infty(\Omega)$. 
	Since the multiplications by $\xi$ and $(1 - \xi)$ are bounded linear operators in $\D$, both $\{\xi v_n\}$ and $\{(1 - \xi) v_n\}$ also converge to zero weakly in $\D$,  $L^{\sigma}(\Omega)$, and $L^{\kappa}(\Omega)$.
	In order to establish the desired compactness of the embedding $\D \hookrightarrow L^{\kappa}(\Omega)$, it is sufficient to prove that both $\{\xi v_n\}$ and $\{(1 - \xi) v_n\}$ converge to zero \textit{strongly} in $L^\kappa(\Omega)$.
	Since $\text{supp}\,(1 - \xi) v_n \subset \Omega_{2\delta}$ and $2<\kappa<\sigma<2^*$, the convergence of $\{(1 - \xi) v_n\}$ follows from \eqref{eq:equivalentnorms}, as $\D$ restricted to functions supported in $\Omega_{2 \delta}$ coincides with $W_0^{1,2}(\Omega_{2 \delta})$ which is compactly embedded in $L^\kappa(\Omega_{2 \delta})$.
	
	Let us investigate the convergence of $\{\xi v_n\}$. Take any $\eta > 0$ and define
	\begin{equation*}
		U_\eta = \left\{ x \in \Omega:~ |\nabla \varphi_1(x)| > \frac \eta 2 \right\}
		\quad \text{and} \quad 
		U'_\eta = \left\{ x \in \Omega:~ |\nabla \varphi_1(x)| < \eta \right\}.
	\end{equation*}
	Thanks to the $C^1$-regularity of $\varphi_1$, both $U_\eta$ and $U'_\eta$ are open and $U_\eta \cup U'_\eta = \Omega$. 
	Since the weakly convergent sequence $\{\xi v_n\}\subset C_0^\infty(\Omega \setminus \Omega_{\delta})$ is bounded in $\D$ and $\varphi_1>0$ in $\overline{\Omega \setminus \Omega_{\delta}}$, we plug $\xi v_n$ into \eqref{eq:47takac} and obtain 
	\begin{equation}\label{eq:Ueta}
		C 
		\int_\Omega |\xi v_n|^\kappa \,dx 
		\leqslant
		C \left ( \int_{U_\eta} |\nabla \varphi_1|^p |\xi v_n|^{\sigma} \,dx \right )^{\frac 1 2}
		+
		C \left ( \int_{U_\eta'} |\nabla \varphi_1|^p |\xi v_n|^{\sigma} \,dx\right )^{\frac 1 2}.
	\end{equation}
	Fix any $\varepsilon>0$. 
	The boundedness of $\{\xi v_n\}$ in $\D$ implies its boundedness in  $L^{\sigma}(\Omega)$,
	and hence there exists $\eta_0 > 0$ such that for any $\eta \in (0,\eta_0)$ we get
	\begin{equation}\label{eq:Ueta1}
		C \left ( \int_{U_{\eta}'} |\nabla \varphi_1|^p |\xi v_n|^{\sigma} \, dx \right )^{\frac 1 2} \leqslant C \eta^{\frac p 2} \|\xi v_n\|_{L^{\sigma}(\Omega)}^\frac{\sigma}{2}\leqslant C \eta_0^{\frac p 2} \|\xi v_n\|_{L^{\sigma}(\Omega)}^\frac{\sigma}{2} \leqslant \varepsilon
	\end{equation}
	for all $n$. 
	Taking any $\eta \in (0,\eta_0)$, let us now prove the existence of $n_0=n_0(\varepsilon)>0$ such that the first term on the right-hand side of \eqref{eq:Ueta} is also less than $\varepsilon$ for any $n \geqslant n_0$. Denote by $w_n$ the restriction of $\xi v_n$ to $U_\eta$, i.e., $w_n = (\xi v_n) |_{U_\eta}$. 
	We have
	\begin{align*}
		\|w_n\|_{W^{1,2}(U_\eta)}^2 
		&= \int_{U_\eta} |\nabla w_n|^2 \,dx + \int_{U_\eta} |w_n|^2 \,dx \\
		&\leqslant \frac{2^{p-2}}{\eta^{p-2}} \int_{U_\eta} |\nabla \varphi_1|^{p-2} |\nabla w_n|^2 \,dx + \int_\Omega |w_n|^2 \,dx \leqslant C \|w_n\|_{\varphi_1}^2.
	\end{align*}
	Thus, the restriction to $U_\eta$ is a bounded linear operator from $\D$ to $W^{1,2}(U_\eta)$. Since bounded linear operators preserve the weak convergence, $\{w_n\}$ converges weakly to zero in $W^{1,2}(U_\eta)$.
	
	By the standard regularity theory, we have
	$\varphi_1 \in C^\infty(\Omega \setminus Z)$, where $Z=\{x \in \Omega: |\nabla \varphi_1|=0\}$.
	Thus, Sard's theorem asserts that the set of critical levels of the function $|\nabla \varphi_1|$ has zero measure. 
	Then, the implicit function theorem guarantees that $\partial U_\eta$ is smooth for almost all $\eta>0$. 
	Perturbing, if necessary, the value of $\eta \in (0,\eta_0)$ fixed above, we deduce that $W^{1,2}(U_\eta)$ is compactly embedded in $L^{\sigma}(U_\eta)$ by the Rellich-Kondrachov theorem, and hence $\{w_n\}$ converges to zero strongly in $L^{\sigma}(U_\eta)$.
	Therefore, we can find the desired $n_0=n_0(\varepsilon)$ such that
	\begin{equation}\label{eq:Ueta2}
		C \left ( \int_{U_\eta} |\nabla \varphi_1|^p |\xi v_n|^{\sigma} \right )^{\frac 1 2} \leqslant C \|\nabla \varphi_1\|_\infty^{\frac p 2}  \left ( \int_{U_\eta} |w_n|^{\sigma} \,dx \right )^{\frac 1 2} \leqslant \varepsilon
	\end{equation}
	for all $n \geqslant n_0$.
	Combining \eqref{eq:Ueta1} and \eqref{eq:Ueta2}, we conclude from \eqref{eq:Ueta} that 
	for any $\varepsilon>0$ there exists $n_1 \geqslant n_0>0$ such that $\|\xi v_n\|_{L^\kappa(\Omega)} \leqslant \varepsilon$ for any $n \geqslant n_1$.
	This completes the proof of the strong convergence of $\{v_n\}$ to zero in $L^{\kappa}(\Omega)$, which gives the compactness of the embedding $\D \hookrightarrow L^{\kappa}(\Omega)$, where $\kappa = \sigma/2+1\in (2,\sigma)$.
	
	\ref{lem:embed:2} For the continuity of the embedding $\D \hookrightarrow W_0^{1,\theta}(\Omega)$ for some $\theta > 1$ in the case $N \geqslant 2$, we refer to  \cite[Corollary~2.8]{brasco-lind} which remains valid for any fixed $1<q<p^*$ with no changes in the proof, as discussed in the proof of \ref{lem:embed:0} above.
	In the case $N=1$ and when $p=q>2$, the claim is given by \cite[Lemma~1.3, p.~238]{FNSS} with an explicit value of $\theta$, and the same argument applies for any $p>2$ and $q>1$.
	
	\ref{lem:embed:3} 
	The continuity of the embedding $\W \hookrightarrow \D$ follows trivially
	from the regularity $\varphi_1 \in C^1(\overline{\Omega})$ 
	and the H\"older inequality.
	\qed

	\bigskip
	\noindent
	\textbf{Acknowledgments.}
	V.~Bobkov was supported by RSF Grant number \href{https://rscf.ru/en/project/22-21-00580/}{22-21-00580}. 
	S.~Kolonitskii was supported by RFBR Grant number 20-01-00630~A.
	The authors are grateful to \textsc{L.~Brasco}, \textsc{A.I.~Nazarov}, and \textsc{P.~Tak\'a\v{c}} for 
	discussions and valuable comments which helped to improve the results and exposition.

\end{document}